\documentclass[11pt,a4paper]{article}
\usepackage{epsf, graphicx}
\usepackage{latexsym,amsfonts,amsbsy,amssymb}
\usepackage{amsmath,amsthm}
\usepackage{cite}
\usepackage{cleveref}
\usepackage{enumerate}
\usepackage{enumitem}
\usepackage{color}
\usepackage{xcolor}   
\usepackage{tikz}
\usetikzlibrary{matrix,chains}

\usepackage{tikz}\usetikzlibrary{positioning} 

\usetikzlibrary{calc}\usetikzlibrary{intersections,through}

\makeatletter

\@addtoreset{equation}{section} \makeatother
\makeatletter \@addtoreset{equation}{section} \makeatother
\setlist[enumerate,1]{label=(\arabic*).,font=\textup,
leftmargin=7mm,labelsep=1.5mm,topsep=0mm,itemsep=-0.8mm}
\setlist[enumerate,2]{label=(\alph*).,font=\textup,
leftmargin=7mm,labelsep=1.5mm,topsep=-0.8mm,itemsep=-0.8mm}

\newtheorem{thm}{Theorem}[section]
\newtheorem*{theorem}{Theorem A}
\newtheorem*{corollary}{Corollary B}
\newtheorem{lem}[thm]{Lemma}

\newtheorem{defi}[thm]{Definition}
\newtheorem{prop}[thm]{Proposition}

\newcommand{\IBr}{{\rm IBr}}
\newcommand{\zs}[1]{\hspace{0.1mm}^{#1}\hspace{-0.4mm}}
\newcommand{\which}{\,|\,}
\newcommand{\Irr}{{\rm Irr}}
\newcommand{\bl}{{\rm bl}}
\newcommand{\Tr}{{\rm Tr}}
\newcommand{\br}{{\rm br}}
\newcommand{\Gal}{{\rm Gal}}
\newcommand{\Aut}{{\rm Aut}}
\newcommand{\hH}{\mathcal{H}}

\newcommand{\tH}{{\rm\bf H}}
\newcommand{\N}{{\rm\bf N}}
\newcommand{\C}{{\rm\bf C}}
\newcommand{\Z}{{\rm\bf Z}}
\newcommand{\Res}{{\rm\bf R}}
\newcommand{\End}{{\rm End}}
\newcommand{\M}{{\rm M}}
\newcommand{\ra}{\rightarrow}
\newcommand{\mK}{\mathcal{K}}
\newcommand{\mO}{\mathcal{O}}
\newcommand{\GL}{{\rm GL}}
\newcommand{\semi}{\rtimes}
\newcommand{\Pj}{\mathcal{P}}
\newcommand{\PQ}{\mathcal{Q}}
\newcommand{\mS}{\mathcal{S}}
\newcommand{\I}{{\rm I}}
\newcommand{\zg}{\unlhd}
\newcommand{\bH}{{\bf H}}
\newcommand{\Fp}{\mathbb{F}_p}
\newcommand{\fE}{\mathfrak{E}}

\renewcommand{\Res}{{\rm Res}}

\begin{document}
\title{\bf An equivariant bijection of irreducible Brauer characters above the Dade--Glauberman--Nagao correspondence}
\date{}
 \author{Qulei Fu\footnote{Shenzhen International Center for Mathematics, Southern University of Science and Technology, Shenzhen 518055, China. E-mail: quleifu@outlook.com. 
 The author is supported by the National Natural Science Foundation of China [Grant No. 12350710787].} }
 \maketitle
\begin{abstract}
The Glauberman correspondence and its generalisation, the Dade--Glauberman--Nagao (DGN) correspondence,  play an important role in studying local-global counting conjectures and their reductions to (quasi-)simple groups.
These reduction theorems require an additional set of compatibility conditions for the DGN correspondence. In this paper, we prove that there exists a bijection of irreducible Brauer characters above the DGN correspondence that is equivariant with Galois automorphisms and group automorphisms and preserves vertices. 
Our proof utilizes the framework of $\hH$-triples developed by Navarro--Späth--Vallejo. The results establish a reduction theorem for the Galois Alperin weight conjecture.

\end{abstract}
\section{Background and main theorem.}\label{Section1}
The Glauberman correspondence and its modular version, the Dade--Glauberman--Nagao (DGN) correspondence, have led to important contributions in the study of group representation theory. This is the case for the research of the McKay conjecture, the Brauer's height zero conjecture and the Alperin weight conjecture, and their reductions to simple groups (see \cite{IMN07,NS14,NT11}). 
Recent work has focused on refining these conjectures in two directions: establishing versions that hold for all finite groups when verified for quasi-simple groups (see \cite{Ro23, MRR23}), and incorporating Galois automorphism \cite{NSV20}. However, such generalizations require additional compatibility conditions for the character bijections constructed above the DGN correspondence.
In this paper, we construct a bijection of irreducible Brauer characters above the DGN correspondence that is equivariant under both group automorphisms and Galois automorphisms. This result provides an essential step in the reduction program for the Galois Alperin weight conjecture as formulated in \cite{FFZ24}. 

Let $\mO$ be a complete discrete valuation ring with  fraction field $\mK$ of characteristic zero and finite residue field $F$ of characteristic $p,$ where $p$ is a fixed prime throughout this paper.
We assume $\mK$ and $F$ are \textit{sufficiently large} for all finite groups under consideration.
Let $\mathcal{H}$ be a group of Galois automorphisms of $\mathcal{K}$ that preserve $\mathcal{O}$, inducing an action on $F$ via the quotient map $\mathcal{O} \to F$. We further assume $\mathcal{H}$ is \textit{large enough} to contain all automorphisms of $F$.

A complex irreducible character $\theta$ of a finite group $G$ has \textit{$p$-defect zero} if $\theta(1)_p = |G|_p$, where $n_p$ denotes the $p$-part of an integer $n$. 
Such characters $\theta$ restrict irreducibly to projective Brauer characters $\theta^\circ$ of $G$. Thus, characters with $p$-defect zero  may be identified with projective irreducible Brauer characters in this setting.

 Let $N$ be a normal subgroup of a finite group $M$ with $M/N$ a $p$-group, and let $\theta$ be an $M$-invariant  irreducible character of $N$ with $p$-defect zero.
 Denote by $\tilde{B}$ be the unique block of $M$ covering the defect zero block $\bl(\theta)$ containing $\theta$, and let $D$ be a defect group of $\tilde{B}.$ 
 By Fong's theorem \cite[Chapter V, Theorem 5.16]{NagTsu}, we have $ND=M$ and $N\cap D=1.$ 
 Consequently $\N_M(D)=D\times \C_N(D).$
The Brauer First Main Theorem guarantees a unique block $B$ of $\N_M(D)$ with defect group $D$ and inducing $\tilde{B}.$
This block $B$ covers a unique defect zero block $b^{\star_D}$ of $\C_N(D)$, containing exactly one irreducible character $\theta^{\star_D}$ with $p$-defect zero. We call $\theta^{\star_D}$  the DGN correspondent of $\theta$ (with respect to $D)$.
The DGN correspondence satisfies uniqueness properties \cite[Chapter V, Theorem 12.1]{NagTsu}.
If $p\nmid|N|$, the character $\theta^{\star_D}$ coincides the $D$-Glauberman correspondent of $\theta,$ showing that the DGN correspondence generalizes the Glauberman correspondence.

Assume that $M$ is normally embedded in a finite group $G$ with $N$ normal in $G$, and that $\theta$ is $G$-invariant. 
 In 1980, Dade first noticed that there is a bijection from the set $\Irr(G\which\theta)$ of irreducible complex characters of $G$ lying over $\theta$ to the set $\Irr(\N_G(D)\which\theta^{\star_D})$ \cite{Dade}. 
 However, a complete proof of Dade's theorem did not appear until Turull's work in 2008 \cite{Tu08}. Prior to this, Puig had outlined an alternative approach for part of the proof \cite{Puig}. 
 Later in 2018, Linckelmann completed Puig's proof of Dade's theorem in his book \cite{Linckelmann}, where he named this result the \emph{Dade's Fusion Splitting Theorem} (see \cite[Theorem 7.9.1]{Linckelmann}).

Turull proved that Dade's bijection can be chosen to be equivariant under the action of both group automorphisms and Galois automorphisms \cite{Tu08}. 
However, he only considered the case of Glauberman correspondence and complex characters. 
In this paper, we generalize Turull's results to the DGN correspondence and Brauer characters. 
Note that Turull's work relies on the theory of Brauer-Clifford groups developed in his series of papers \cite{Tu94,Tu09:BC,Tu09:FM}. 
It appears that the language of $\hH$-triples introduced by Navarro-Sp\"ath-Vallejo in \cite{NSV20} is better suited for describing character bijections in the related groups.

   We require that the elements in $\hH$ are composed from the left, and we use right exponential notation (i.e. $x^{\sigma}$ for the image of $x\in\mK\sqcup F$ under the automorphism $\sigma\in\hH$). 
   We also write $h^g=g^{-1}hg$ for group elements $g,h\in G.$
   For blocks and characters in general we use the notation as introduced in \cite{NagTsu}.
Let $S$ be a subgroup of a finite group $G$ and $\chi$ a Brauer character of $S$. 
For any $\sigma \in \mathcal{H}$ and $g\in G,$  the function $\chi^{\sigma}:S_{p'}\to \mO$ defined by $h \mapsto \chi(h)^{\sigma}$ is a Brauer character of $S$,
and the function $\chi^{g}: S^{g}_{p'} \to \mO$ defined by $h \mapsto \chi(h^{g^{-1}})$ is a Brauer character of $S^{g}$, where $S_{p'}$ denotes the set of $p$-regular elements of $S$.

Let $N$ be a subgroup of a finite group $G$ and $\theta\in\IBr(N).$ Then we denote by $\IBr(G\which \theta)$ the subset of $\IBr(G)$ consisting of elements lying over $\theta.$  We denote by $\theta^{\hH}$ the $\hH$-orbit of $\theta$ and by $\IBr(G\which \theta^{\hH})$ the subset of $\IBr(G)$ consisting of elements lying over some $\hH$-conjugate of $\theta.$ This set is $$\bigcup_{\sigma\in\hH}\IBr(G\which \theta^{\sigma}).$$
We denote by $\bl(\theta)$ the block of $N$ containing $\theta$.
If $\chi\in\IBr(G),$ we denote by $\chi_N$ the restriction of $\chi$ to $N.$
A vertex of  $\chi$ is a minimal $p$-subgroup $V$ of $G$ for which the associated simple $FG$-module is relatively $V$-projective (see \cite[Chapter IV, Section 3]{NagTsu} for details).
We denote by $\mathcal{S}(G,N)$  the set of subgroups of $G$ containing $N.$

 We are going to state the main theorem of this paper. 

 \begin{theorem}\label{thmA}
   Let $N$ be a normal subgroup of a finite group $G,$ with $M/N$ a normal $p$-subgroup of $G/N.$
   Let $\theta\in\IBr(N)$ be an $M$-invariant character with $p$-defect zero, and let $D$ be a defect group of the unique block of $M$ covering $\bl(\theta)$.
   Set $C=\C_N(D)$ and $H=\N_G(D),$ and assume that $\theta^{\hH}$ is $G$-stable.
   Let $\varphi=\theta^{\star_D}\in\IBr(C)$ be the DGN correspondent of $\theta.$
Then there exists an $\hH$-equivariant bijection 
   $$\varDelta_S: \IBr( S\which\theta^{\hH})\ra \IBr(S\cap H\which\varphi^{\hH})$$ for each $S\in\mathcal{S}(G,N),$ satisfying:
   \begin{enumerate}
        \item \textbf{Conjugation Compatibility}: For any $\chi \in \IBr(S \which \theta^{\hH})$ and $g \in H,$ 
        \[
        \varDelta_S(\chi)^g = \varDelta_{S^g}(\chi^g).
        \]
        
        \item \textbf{Restriction Compatibility}: For any $\chi \in \IBr(S \which \theta^{\hH})$ and $T \in \mathcal{S}(S,N),$ 
        \[
        \varDelta_S(\chi)_{T \cap H} = \varDelta_T(\chi_T).
        \]
        
        \item \textbf{DGN Consistency}: $\varDelta_N(\theta) = \varphi$.
        
        \item \textbf{Vertex Relation}: For any $\chi \in \IBr(S \which \theta^{\hH})$ with a vertex $V_1$, there exists a vertex $V_2$ of $\varDelta_S(\chi)$ such that $V_1N = V_2N$.
    \end{enumerate}
 \end{theorem}

\noindent {\itshape Remark.} It is straightforward to see that $G=HN$ and $C=H\cap N$ in the theorem. 
Furthermore, we observe that the bijection $\varDelta_T$ in condition (2) naturally extends to an $\mathbb{N}$-linear map
\[
\varDelta_T \colon \mathbb{N}\IBr(T \which \theta^{\hH}) \to \mathbb{N}\IBr(T\cap H \which \varphi^{\hH}),
\]
where $\mathbb{N}$ denotes the set of non-negative integers. 
This extension is necessary because the restriction $\chi_T$ may be reducible for $\chi \in \IBr(S \which \theta^{\hH})$.

By Theorem A we deduce the following corollary, which says that there exists a bijection of irreducible Brauer characters above the DGN correspondence that is equivariant under the action of both Galois automorphisms and group automorphisms. 

\begin{corollary}
Let $N$ be a normal subgroup of a finite group $G,$ with $M/N$ a normal $p$-subgroup of $G/N.$
   Let $\theta\in\IBr(N)$ be a $G$-invariant character with $p$-defect zero, and let $D$ be a defect group of the unique block of $M$ covering $\bl(\theta)$.
   Let $\varphi=\theta^{\star_D}\in\IBr(\C_N(D))$ be the DGN correspondent of $\theta$ with respect to $D.$
Then there exists an $\big(\hH\times\Aut(G)\big)_{\theta,D}$-equivariant bijection 
   $$f: \IBr( G\which\theta )\ra \IBr(\N_G(D)\which\varphi),$$ 
  such that if $V_1$ is a vertex of $\chi\in\IBr( G\which\theta)$, then there exists a vertex $V_2$   of $f(\chi)$ such that $V_1N =V_2N.$ 
\end{corollary}

The proof of Theorem~A employs the theory of $\hH$-triples developed in \cite{NSV20}, adapted here to the modular setting. Our strategy proceeds in three stages.
In Section~2, we show that when two modular $\hH$-triples satisfy a partial order relation analogous to \cite[Definition~1.5]{NSV20}, one can construct a system of character bijections between related groups that fulfills all conditions of Theorem~A.
In Section~3, 
to verify this partial order relation, we associate to each modular $\hH$-triple a 2-th cohomology class of the quotient group. We prove that $\hH$-triples sharing a common cohomology class necessarily satisfy the required partial order relation.
In Section~4,
using Dade's fusion splitting theorem, we demonstrate that the $\hH$-triples in Theorem~A indeed share such a cohomology class, thereby completing the proof. The section concludes with the proof of Corollary~B.

\section{Modular $\hH$-triples} 
The partial order relation of character triples was first introduced in \cite{NS14}. To incorporate Galois automorphisms, Navarro, Sp\"ath, and Vallejo \cite{NSV20} developed the concept of $\hH$-triples and extended this partial order to $\hH$-triples.
In this paper, we work with the modular version of $\hH$-triples. 
Note that the partial order relation of modular character triples are already exists in the literature \cite{MRR23, SV16}. 
For foundational material on character triples and their partial order relations, we refer to \cite[Chapter 10]{Navarro:McKay} and \cite[Chapter III, Section 5]{NagTsu}.

Throughout this paper, we fix a prime $p$ and work with a $p$-modular system $(\mK,\mO,F)$ as defined in Section~1 (where $\hH$ is also defined). 

For any subgroup $S\leqslant G$ and Brauer character $\chi$ of $S$, and for any $a = (\sigma, g) \in \hH \times G$, we define $\chi^a$ as the Brauer character of $S^g$ satisfying
\[
\chi^a(h) = \chi(ghg^{-1})^\sigma \quad \text{for all } h \in S_{p'}^g.
\]
For consistency, we set $S^a := S^g$, making $\chi^a$ a Brauer character of $S^a$.

Let $X, X' \colon G \to \GL_m(F)$ be group representations (or projective representations). We say they are \emph{similar} (written $X' \sim X$) if there exists an invertible matrix $T \in \GL_m(F)$ such that
\[
X'(g) = TX(g)T^{-1} \quad \text{for all } g \in G.
\]

Throughout this paper, all tensor products are over $F$ unless specified otherwise, and all algebras and modules are finite-dimensional.

Let $N$ be a normal subgroup of a finite group $G,$ and $\theta\in\IBr(N).$ 
Assume the $\hH$-orbit $\theta^{\hH}$ is $G$-stable. 
Then we write $(G,N,\theta)_{\hH}$ and call it a \emph{modular $\hH$-triple}. 
Since all $\hH$-triples considered in this paper are modular, we will simply refer to them as $\hH$-triples when no confusion arises.

Let $(G,N,\theta)_{\hH}$ be an $\hH$-triple. For any $g \in G$, there exists $\sigma \in \hH$ such that $\theta^g = \theta^\sigma$. Let $G_\theta$ denote the stabilizer of $\theta$ in $G$, which is normal in $G$ since for any $h \in G_\theta$ and $g \in G$,
\[
\theta^{ghg^{-1}} = \theta^{\sigma g^{-1}} = \theta.
\]
By \cite[Chapter V, Theorem 5.7]{NagTsu}, there is a projective representation $\Pj \colon G_\theta \to \GL_{\theta(1)}(F)$ associated with $\theta$.
The corresponding factor set $\alpha \colon G_\theta \times G_\theta \to F^\times$, defined by
\[
\Pj(h)\Pj(k) = \alpha(h,k)\Pj(hk) \quad (h,k \in G_\theta),
\]
is constant on $N \times N$-cosets,
and may therefore be viewed as a factor set of $G_\theta/N$.
Note that $\alpha(1,1)=1.$
For $a = (\sigma,g) \in (\hH \times G)_\theta$ (the stabilizer of $\theta$ in $\hH \times G$), the map 
\[
\Pj^a \colon G_\theta \to \GL_{\theta(1)}(F),
h\mapsto \Pj(ghg^{-1})^\sigma,
\]
is also a projective representation associated with $\theta$. By \cite[Remark 1.3, Lemma 1.4]{NSV20}, there exists a unique function $\mu_a \colon G_\theta \to F^\times$ such that $\Pj^a \sim \mu_a \Pj$ (i.e. there exists an invertible matrix $T \in \GL_{\theta(1)}(F)$ such that $\Pj^a(h) = \mu_a(h)T\Pj(h)T^{-1}$ for all $h \in G_\theta$).
We say $\mu_a$ is determined by $\Pj^a \sim \mu_a \Pj$.
Since $\mu_a$ is constant on $N$-cosets,
we can regard $\mu_a$ as a function on $G_\theta/N$. 
Note that $\mu_a(1)=1.$

The following definition is analogous to \cite[Definition 1.5]{NSV20}. A distinction here is that the two projective representations are not required to coincide on the centralizer subgroup.

\begin{defi}\label{def:H-tri:ord}
  Suppose that $(G,N,\theta)_{\hH}$ and $(H,C,\varphi)_{\hH}$ are $\hH$-triples. We write $(G,N,\theta)_{\hH}\geqslant (H,C,\varphi)_{\hH}$ if the following conditions hold:
\begin{enumerate}
  \item $G=NH$ and $N\cap H=C.$
  \item $(\hH\times H)_{\theta}=(\hH\times H)_{\varphi}.$ In particular, $H_{\theta}=H_{\varphi}.$
  \item There are projective representations $\Pj$ of $G_{\theta}$ and $\Pj'$ of $H_{\varphi}$ associated with $\theta$ and $\varphi$ with factor sets $\alpha$ and $\alpha',$ respectively, such that $\alpha'=\alpha_{H_{\theta}\times H_{\theta}}.$
  \item For any $a\in(\hH\times H)_{\theta},$ the  functions $\mu_a$ and $\mu'_a$ determined by $\Pj^a\sim\mu_a\Pj$ and $\Pj'^a\sim\mu_a'\Pj',$ respectively, agree on $H_{\theta}.$ 
\end{enumerate}
\end{defi}

In the situation described above we say that $(\Pj,\Pj')$ \emph{gives} $(G,N,\theta)_{\hH}\geqslant (H,C,\varphi)_{\hH}.$

Now, let $G$ be a finite group, $N$ a normal subgroup of $G,$ and $C,H$ subgroups of $G$ such that $G=NH,N\cap H=C.$ 
Let $\theta\in\IBr(N)$ be $G$-invariant, and $\varphi\in\IBr(C)$ be $H$-invariant.
Assume that there exist projective representations $\Pj$ of $G$ and $\Pj'$ of $H$ associated with $\theta$ and $\varphi$ with factor sets $\alpha$ and $\alpha',$ respectively, such that $\alpha'=\alpha$ when they are viewed as factor sets of the group $G/N=H/C.$ 
Then we can construct a bijection $$':\IBr(G\which \theta)\ra \IBr(H\which\varphi),$$ such that if the group representation $\PQ\otimes \Pj$ affords the Brauer character $\chi\in\IBr(G\which \theta),$ then the group representation $\PQ\otimes \Pj'$ affords the Brauer character $\chi',$ where $\PQ$ is any (absolutely) irreducible projective representation of $G/N=H/C$ with factor set $\alpha^{-1}$ (see \cite[Theorem 10.13]{Navarro:McKay} or \cite[Chapter V, Theorem 5.8]{NagTsu} for details). In this case, we say that the pair $(\Pj,\Pj')$ \emph{induces} the bijection $'.$

Keep the notation above. Let $a=(\sigma_a,\phi_a)\in\big(\hH\times\Aut(G)\big)_{N,H,\theta,\varphi},$ the subgroup of $\hH\times\Aut(G)$ consisting of elements stabilizing $N,H,\theta$ and $\varphi$ simultaneously. Define 
\begin{align*}
  \Pj^a: & G\ra\GL_{\theta(1)}(F),g\mapsto \Pj(g^{\phi_a^{-1}})^{\sigma_a},\quad \text{ and } \\
  \Pj'^a: & H\ra \GL_{\varphi(1)}(F), h\mapsto \Pj'(h^{\phi_a^{-1}})^{\sigma_a}.
\end{align*}
Then $\Pj^a$ and $\Pj'^a$ are projective representations of $G$ and $H$ associated with $\theta$ and $\varphi$, respectively. Let $\mu_a,\mu'_a:G/N=H/C\ra F^{\times}$ be the  functions determined by $\Pj^a\sim\mu_a\Pj$ and $\Pj'^a\sim\mu'_a\Pj',$ respectively.

\begin{lem}\label{lem:H-tri1}
  With the previous notation. If $\mu_a=\mu'_a$ when viewed as maps from $G/N=H/C$ to $F^{\times},$ then $(\chi^a)'=(\chi')^a$ for every $\chi\in\IBr(G\which\theta),$ where $\chi^a\in\IBr(G\which\theta)$ is defined by $\chi^a(g)=\chi(g^{\phi_a^{-1}})^{\sigma_a}$ for $g\in G_{p'}.$
\end{lem}
\begin{proof}
  Let $\chi\in\IBr(G\which\theta),$ and let $\PQ$ be a projective representation of $G/N=H/C$ with factor set $\alpha^{-1}.$ 
  Suppose that $\PQ\otimes\Pj$ affords $\chi,$ then  $\PQ\otimes\Pj'$ affords $\chi'.$ Since $(\PQ\otimes\Pj)^a$ affords $\chi^a$, and 
  $$(\PQ\otimes\Pj)^a= \PQ^a\otimes\Pj^a\sim\PQ^a\otimes\mu_a\Pj= \PQ^a\mu_a\otimes\Pj,$$
  If follows that $\PQ^a\mu_a\otimes\Pj'$ affords $(\chi^a)'$. 
  On the other hand, 
  $$\PQ^a\mu_a\otimes\Pj'= \PQ^a\mu'_a\otimes\Pj'= \PQ^a\otimes\mu'_a\Pj'\sim\PQ^a\otimes(\Pj')^a= (\PQ\otimes\Pj')^a,$$
  and $(\PQ\otimes\Pj')^a$ affords  $(\chi')^a$. 
  Thus, $(\chi^a)'=(\chi')^a.$
\end{proof}

In fact, the bijection $':\IBr(G\which \theta)\ra \IBr(H\which \varphi)$ preserves vertices of characters. 
To prove this, we require some preliminaries on $G$-algebras.

Let $A$ be an $F$-algebra, and let $\Aut_F(A)$ denote the group of $F$-algebra automorphisms of $A,$ where automorphisms are composed from the left.
Let $G$ be a finite group. A $G$-algebra (or more precisely a $G$-algebra over $F$) is a pair $(A,\psi),$ where $A$ is an $F$-algebra and $\psi:G\ra\Aut_F(A)$ is a group homomorphism. If no confusion arises we only write $A$ instead of $(A,\psi)$ to denote a $G$-algebra. The (right) action of $g\in G$ on $x\in A$ is denoted by $x^g:=\psi(g)(x).$ 
For further details, see \cite[\S 10]{Thevenaz}.

Let $A$ be a $G$-algebra over $F.$ For any subgroup $S$ of $G,$ we define $$A^S=\{x\in A\which x^g=x \text{ for any } g\in S\}.$$
 If $T$ is a subgroup of $S,$ then the \emph{trace map} $\Tr_T^S$ is defined by $$\Tr_T^S:A^T\ra A^S,x\mapsto x^{g_1}+\cdots+x^{g_t},$$ where ${g_1,\cdots,g_t}$ is a complete set of representatives of right $T$-cosets in $S.$ 
 The trace map $\Tr_T^S$ is $F$-linear and is independent of the choice of coset representatives of $T$ in $S$.
We denote its image by $A_T^S.$
 
 Let $\M_s(F)$ denote the matrix algebra of degree $s$ over $F$, where $s$ is a positive integer, and let $\I_s$ be the identity matrix of degree $s$. Note that the general linear group $\GL_s(F)$ coincides with the group of units $\M_s(F)^\times$ of $\M_s(F)$.

\begin{lem}\label{lem:H-tri2}
Let $X: G \to \GL_s(F)$ be either an irreducible group representation or an irreducible projective representation of degree $s$. 
Let $A = \M_s(F)$, where $\GL_s(F)$ is identified with its unit group. 
Then $A$ becomes a $G$-algebra under the conjugation action:
$$x^g=X(g)^{-1}xX(g)\quad \text{for } x\in A,g\in G.$$
There exists a $p$-subgroup $V \leqslant G$ such that:
\begin{enumerate}
\item $1_A\in A^G_V.$
\item For any subgroup $H \leqslant G$ with $1_A \in A^G_H$, $V$ is contained in some $G$-conjugate of $H$.
\end{enumerate}  
\end{lem}
\begin{proof}
  For the case when $X$ is a group representation, the results follows directly from \cite[Chapter IV, Theorem 2.2 and 3.3]{NagTsu}.
  When $X$ is an irreducible projective representation, $A$ is a {primitive} $G$-algebra in the sense of \cite[pp. 102]{Thevenaz}. 
  In this setting, any defect group V of A (as defined in \cite[§18]{Thevenaz}) satisfies the required conditions by \cite[Lemma 18.2]{Thevenaz}.
\end{proof}

 When $X$ is an irreducible group representation,
 Theorems 2.2 and 3.3 in Chapter IV of \cite{NagTsu} establish that $V$ is precisely a vertex of the simple $FG$-module corresponding to $X$. 
 By extension, we call $V$ a vertex of the representation $X$ itself. This terminology naturally carries over to the case when $X$ is an irreducible projective representation. 
 Lemma \ref{lem:H-tri2} shows that in both cases (ordinary or projective), the vertex $V$ is uniquely determined up to $G$-conjugacy.

\begin{prop}\label{prop:H-triVt}
  Keep the notation in the bijection $':\IBr(G\which \theta)\ra\IBr(H\which \varphi)$. 
  If $V_1$ is a vertex of $\chi\in\IBr(G\which\theta),$ then there exists a vertex $V_2$ of $\chi'$ such that $V_1N=V_2N.$
\end{prop}
\begin{proof}
  Let $\bar{G}=G/N,$ and for $g\in G,$ denote its image in $\bar{G}$ by $\bar{g}.$
  Let $\PQ:\bar{G}\ra\GL_s(F)$ be a projective representation with factor set $\alpha^{-1}$ such that the group representation $$\PQ\otimes\Pj:G\ra \GL_{s\theta(1)}(F),g\mapsto \PQ(\bar{g})\otimes \Pj(g)$$ affords $\chi.$ 
  Note that $\GL_s(F)\otimes \GL_{\theta(1)}(F)$ is a subgroup of $\GL_{s\theta(1)}(F)$. 
  Let $$Q=\M_s(F)\quad \text{and}\quad P=\M_{s\theta(1)}(F)=\M_s(F)\otimes \M_{\theta(1)}(F).$$ And we regard $\GL_s(F)$ and $\GL_{s\theta(1)}(F)$ as the unit group of $Q$ and $P,$ respectively. 
  Then the functions $\PQ:\bar{G}\ra Q^{\times}$ and $\PQ\otimes\Pj:G\ra P^{\times}$ induce $\bar{G}$-algebra and $G$-algebra structures on $Q$ and $P,$ respectively, via conjugation actions. 
  For any $T\in\mS(G,N),$ define the function $f_T:Q^{\bar{T}}\ra P^T,x\mapsto x\otimes \I_{\theta(1)}.$

  {\itshape Step 1.} For any $T\in\mS(G,N),$ the function $f_T:Q^{\bar{T}}\ra P^T,x\mapsto x\otimes \I_{\theta(1)}$ is an isomorphism of $F$-vector spaces, and the following diagram commutes:
   
  
   $$\begin{tikzpicture}
\node (1) at(0,0) {$Q^{\bar{G}}$};
\node  (2) at(2,0) {$P^G$};
\node  (3) at(0,-2) {$Q^{\bar{T}}$};
\node  (4) at(2,-2) {$P^T$};
\draw[->](3)--(1) ;
\draw[->]  (4)--(2) ;
\draw[->] (1)--(2);
\draw[->] (3)--(4);
\coordinate [label=left:$\Tr_{\bar{T}}^{\bar{G}}$] (A) at (0,-1);
\coordinate [label=right:$\Tr_T^G$] (B) at (2,-1);
\coordinate [label=above:$f_G$] (C) at (1,0);
\coordinate [label=above:$f_T$] (D) at (1,-2);
\end{tikzpicture}.$$

First, for any $x\in Q^{\bar{T}},$ and any $g\in T,$ 
\begin{align*}
  &(\PQ\otimes \Pj)(g)^{-1}
  \cdot (x\otimes \I_{\theta(1)})
  \cdot (\PQ\otimes \Pj)(g)\\
    =&(\PQ(\bar{g})\otimes \Pj(g))^{-1}
\cdot (x\otimes \I_{\theta(1)})
\cdot (\PQ(\bar{g})\otimes \Pj(g)) \\
   =& (\PQ(\bar{g})^{-1} x  \PQ(\bar{g})) 
   \otimes (\Pj(g)^{-1}\I_{\theta(1)}\Pj(g))
   =x\otimes \I_{\theta(1)}.
\end{align*}
Thus $f_T(x)\in P^T.$ 
The function $f_T$ is clearly an injective $F$-linear homomorphism.
To prove it is surjective, we show any $y \in P^T$ has the form $y=x\otimes \I_{\theta(1)}$ for some $x\in \PQ^{\bar{T}}.$
Since the set $\{\I_s\otimes \Pj(n)\which n\in N\}$ generates the subalgebra $\I_s\otimes \M_{\theta(1)}(F),$ we conclude from \cite[Chapter II, Lemma 4.1]{NagTsu} that $$P^N=\C_P(\I_s\otimes \M_{\theta(1)}(F))
=\M_s(F)\otimes \I_{\theta(1)}.$$ 
Therefore any $y\in P^T\subseteq P^N$ admits a decomposition $y=x\otimes I_{\theta(1)}$ for some $x\in\M_s(F).$
As $y$ is $T$-invariant, we see $x$ is $\bar{T}$-invariant.
Finally, the equation $f_G\cdot \Tr_{\bar{T}}^{\bar{G}}=\Tr_T^G\cdot f_T$ follows by direct computation, which completes the proof of this step.

Let $V \in \mathcal{S}(G,N)$ be a minimal subgroup satisfying $1_P \in P_V^G$. Through the isomorphisms $\{f_T \which T \in \mathcal{S}(G,N)\}$, we conclude that $\overline{V}$ is the minimal subgroup of $\overline{G}$ with $1_Q \in Q_{\overline{V}}^{\overline{G}}$. By Lemma \ref{lem:H-tri2}, this shows that $\overline{V}$ is a vertex of the projective representation $\PQ.$
Let $V_1$ be a minimal subgroup of $V$ such that $1_P\in P_{V_1}^G.$ Then $V_1$ is a vertex of $\chi.$  
We want to prove that $V=NV_1.$ 
If $NV_1\lneqq V,$ then we can deduce from $1_P\in P^G_{V_1}$ that $1_P\in P^G_{NV_1}.$ 
But this contradicts the definition of $V.$ Thus we have proved $V=NV_1.$

We have proved that the vertex $V_1$ of $\chi$ and the vertex $\bar{V}$ of $\PQ$ satisfy $\bar{V}=V_1N/N.$ 
In a similar way we can prove that the vertex $V_2$ of $\chi'$ satisfies $\bar{V}=V_2C/C.$ That is, $V_1N=V_2N,$ as desired.
\end{proof}

The following result strengthens \cite[Theorem 1.10]{NSV20}.
A difference in our case is that we need to consider the action of group automorphisms.

\begin{thm}\label{thm:H-tri}
   Let $(G,N,\theta)_{\hH}$ and $(H,C,\varphi)_{\hH}$ be $\hH$-triples such that $(G,N,\theta)_{\hH}\geqslant (H,C,\varphi)_{\hH}.$ Then for any $S\in\mS(G,N),$ there exists an $\hH$-equivariant bijection 
   $$\varDelta_S: \IBr( S\which\theta^{\hH})\ra \IBr(S\cap H\which\varphi^{\hH}),$$ 
  such that the following conditions hold.
   \begin{enumerate}
     \item For any $\chi\in\IBr(S\which\theta^{\hH})$ and $g\in H,$ we have $\varDelta_S(\chi)^g= \varDelta_{(S^g)}(\chi^g).$ 
     \item For any  $\chi\in\IBr(S\which \theta^{\hH})$ and  $T\in\mathcal{S}(S, N)$,
           $$\varDelta_S(\chi)_{T\cap H}= \varDelta_T(\chi_T).$$   
     \item $\varDelta_N(\theta)=\varphi.$
     \item Let $\chi\in\IBr( S\which\theta^{\hH})$, and let $V_1$ be a vertex of $\chi.$ Then there exists a vertex $V_2$   of $\varDelta_S(\chi)$ such that $V_1N=V_2N.$ 
   \end{enumerate}
 \end{thm}
\begin{proof}
  Suppose that $(\Pj,\Pj')$ gives $(G,N,\theta)_{\hH}\geqslant (H,C,\varphi)_{\hH}.$
  The construction of character bijections $\Delta_S$ proceeds through three steps.

  {\itshape Step 1}. Let $\{S_1=N,S_2,\cdots,S_l\}$ be a complete set of representatives for the $H$-conjugacy classes of subgroups in $\mS(G_{\theta},N).$ 
  For each $1\leqslant i\leqslant l,$ let $\Delta_{S_i}^1:\IBr(S_i\which\theta)\ra\IBr(S_i\cap H\which \varphi) $ be the bijection induced by $(\Pj_{S_i},\Pj_{S_i\cap H}').$ 
  (Here $\Pj_{S_j}$ is the restriction of $\Pj$ to $S_j.$)
  This bijection  $\Delta_{S_i}^1$ is $(\hH\times \N_H(S_i))_{\theta}$-equivariant.

  For any $a\in (\hH\times H)_{\theta},$ let $\mu_a:G_{\theta}\ra F^{\times}$ and $\mu_a':H_{\theta}\ra F^{\times}$ be the functions determined by $\Pj^a\sim\mu_a\Pj$ and $\Pj'^{a}\sim\mu_a'\Pj',$ respectively. 
  Notice that $\mu_a $ and $\mu_a'$ agree on  $H_{\theta}$ by the assumption.
  Now fix any $i,$ and  let  $a\in(\hH\times\N_H(S_i))_{\theta}.$ 
  Since $(\Pj_{S_i})^a\sim\mu_{a,S_i}\Pj_{S_i}$ and $(\Pj'_{S_i\cap H})^a\sim\mu'_{a,S_i\cap H}\Pj'_{S_i\cap H},$ where $\mu_{a,S_i}$ is the restriction of $\mu_a$ to $S_i,$ by Lemma \ref{lem:H-tri1} we have that the bijection $\Delta_{S_i}^1$ is $(\hH\times \N_H(S_i))_{\theta}$-equivariant. 
  
   {\itshape Step 2}. We give the definition of $\varDelta_S:\IBr(S\which \theta^{\hH})\ra\IBr(S\cap H\which \varphi^{\hH})$ when $S\in\mS(G_{\theta},N).$ Let $S\in\mS(S_{\theta},N).$ 
   For any $\chi\in\IBr(S\which \theta^{\hH}),$ define $\varDelta_S(\chi)=\varDelta_{S^a}^1(\chi^a)^{a^{-1}},$ where $a\in\hH\times H$ is such that $S^a\in\{S_1,\cdots,S_l\}$ and $\chi^a\in\IBr(S^a\which \theta).$ Then $\Delta_S$ is a well-defined bijection, and for any $a\in\hH\times H$ we have $\varDelta_S(\chi)^a=\varDelta_{S^a}(\chi^a).$ 
   
   Let $S\in\mS(G_{\theta},N),$ and $\chi\in\IBr(S\which \theta^{\hH}).$ Then there exists $a\in\hH\times H$  such that  $S^a\in\{S_1,\cdots,S_l\}$ and $\chi^a\in\IBr(S^a\which \theta).$ 
   We now prove that $\varDelta_S(\chi):=\varDelta_{S^a}^1(\chi^a)^{a^{-1}}
   \in\IBr(S'\which \varphi^{\hH})$ is independent of the choice of $a.$ Suppose that $a'\in\hH\times H$ is another element such that  $S^{a'}\in\{S_1,\cdots,S_l\}$ and $\chi^{a'}\in\IBr(S^{a'}\which \theta).$ 
   Since any different two subgroups in $\{S_1,\cdots,S_l\}$ are not $H$-conjugate, we see $S^a=S^{a'}.$
   Thus we can let $a'=ax,$ where $x\in\hH\times H$ stabilizes $S^a.$
   Now $\chi^a,\chi^{ax}\in\IBr(S^a\which \theta),$ we can deduce that $\theta^x=\theta,$ thus $x\in\big(\hH\times \N_H(S^{a})\big)_{\theta}.$ 
   Since the bijection $\varDelta_{S^a}^1$ is $\big(\hH\times \N_H((S^a)')\big)_{\theta}$-equivariant, we have
   $$\varDelta_{S^a}^1(\chi^{a'})^{a'^{-1}}
   =\varDelta_{S^a}^1(\chi^{ax})^{x^{-1}a^{-1}}
   =\varDelta_{S^a}^1(\chi^a)^{a^{-1}}.$$
   Thus $\varDelta_S$ is well-defined.
   In order to show that $\varDelta_S$ is a bijection, we can construct the inverse map $\Delta_S':\IBr(S\cap H\which \varphi^{\hH})\ra\IBr(S\which \theta^{\hH})$ in a similar way like $\varDelta_S,$ and prove that $\varDelta_S\varDelta'_S$ and $\varDelta'_S\varDelta_S$ are both identity maps.
  What we left is to prove that $\varDelta_S(\chi)^a=\varDelta_{S^a}(\chi^a)$ holds for any $S\in\mS(G_{\theta},N),\chi\in\IBr(S\which\theta^{\hH})$ and $a\in\hH\times H.$ 
   Let $y\in\hH\times H$ be such that $S^y\in\{S_1,\cdots,S_l\}$ and $\chi^y\in\IBr(S^y\which \theta).$ Then by definition we have $\varDelta_S(\chi)=\varDelta_{S^y}^1(\chi^y)^{y^{-1}}.$ 
   Notice that $\chi^a\in\IBr(S^a\which\theta^{\hH})$ and $a^{-1}y\in\hH\times H$ is such that $S^{a(a^{-1}y)}=S^y\in\{S_1,\cdots,S_l\}$ and $(\chi^a)^{a^{-1}y}=\chi^y\in\IBr(S^y\which\theta),$ thus by definition we have $\varDelta_{S^a}(\chi^a) =\varDelta^1_{S^y}(\chi^y)^{y^{-1}a}.$
   Since $\varDelta^1_{S^y}(\chi^y)^{y^{-1}a} =\Delta_S(\chi)^a,$ we have $\varDelta_{S^a}(\chi^a)=\Delta_S(\chi)^a.$
   
   {\itshape Step 3}. We give the definition of $\varDelta_S:\IBr(S\which \theta^{\hH})\ra\IBr(S\cap H\which \varphi^{\hH})$ for any $S\in\mS(G,N).$ Let $S\in\mS(S,N)$ and $\chi\in\IBr(S\which \theta^{\hH}).$
   Choose $\theta_1$ to be an irreducible constituent of $\chi_N,$ thus $\theta_1\in\theta^{\hH}.$ Let $\chi_1\in\IBr(S_{\theta}\which \theta_1)$ be the Clifford correspondent of $\chi$ (here $S_{\theta}$ is the stabilizer of $\theta$ in $S,$ and notice that $S_{\theta}=S_{\theta_1}\zg S$).  We define $\varDelta_S(\chi)=\varDelta_{S_{\theta}}(\chi_1)^{S\cap H}.$ 
 Then $\Delta_S$ is a well-defined bijection, and for any $S\in\mS(G,N),\chi\in\IBr(S\which \theta^{\hH})$ and $a\in\hH\times H$ we have $\varDelta_S(\chi)^a=\varDelta_{S^a}(\chi^a).$ 
 
  Firstly, we prove that $\varDelta_S(\chi):=\varDelta_{S_{\theta}}(\chi_1)^{S\cap H}$ is independent of the choice $\theta_1$. Suppose that $\theta_2$ is another irreducible constituent of $\chi_N.$ By Clifford theory, there exists $g\in S\cap H$ such that $\theta_2=\theta_1^g.$
   As $\chi_2:=\chi_1^g\in\IBr(S_{\theta}\which \theta_2)$ is such that $\chi_2^S=\chi,$ we see that $\chi_2\in\IBr(S_{\theta}\which \theta_2)$ is the Clifford correspondent of $\chi.$ 
   We need to prove that $\varDelta_{S_{\theta}}(\chi_2)^{S'}
   =\varDelta_{S_{\theta}}(\chi_1)^{S'}.$
   Since
   $$\varDelta_{S_{\theta}}(\chi_2)^{S\cap H}
   =\varDelta_{S_{\theta}}(\chi_1^g)^{S\cap H}
   =(\varDelta_{S_{\theta}}(\chi_1)^g)^{S\cap H}
   =\varDelta_{S_{\theta}}(\chi_1)^{S\cap H},$$
where the second equation comes from results in {\itshape Step 2,}
   we see $\varDelta_S:\IBr(S\which\theta^{\hH})
   \ra\IBr(S\cap H\which\varphi^{\hH})$ is well-defined. 
   Notice that when $S\in\mS(G_{\theta},N),$ the definition of $\varDelta_S$ is the same with the definition in {\itshape Step 2,} thus there is no confusion with the notation.
   To show $\Delta_S$ is a bijection, we can construct the inverse map $\varDelta'_S:\IBr(S\cap H\which\varphi^{\hH})
   \ra\IBr(S\which\theta^{\hH})$ in a similar way like $\varDelta_S$ and prove that $\varDelta_S\varDelta'_S$ and $\varDelta'_S\varDelta_S$ are identity maps.
   Now we prove for any $a\in\hH\times H$ we have  $\varDelta_S(\chi)^a=\varDelta_{S^a}(\chi^a).$
   Notice that $(\theta_1)^a$ is an irreducible constituent of $(\chi^a)_N,$ and $(\chi_1)^a\in\IBr(S_{\theta}^a\which \theta^a_1)$ is the Clifford correspondent of $\chi.$ Thus by definition we have $\varDelta_{S^a}(\chi^a) =\varDelta_{S^a_{\theta}}(\chi_1^a)^{S^a\cap H}.$ 
   Since $$\varDelta_{S^a_{\theta}}(\chi_1^a)^{S^a\cap H}
   =(\varDelta_{S_{\theta}}(\chi_1)^a)^{S^a\cap H}
   =(\varDelta_{S_{\theta}}(\chi_1)^{S\cap H})^a
   =\varDelta_S(\chi)^a,$$
   where in the second equation we need to notice that $(S\cap H)^a=S^a\cap H,$ we have $\varDelta_{S^a}(\chi^a)=\varDelta_S(\chi)^a.$
   
   Now we have given the definition of $\varDelta_S$ for any $S\in\mS(G,N),$ and have proved that $\varDelta_S(\chi)^a=\varDelta_{S^a}(\chi^a)$ holds for any $\chi\in\IBr(S\which \theta^{\hH})$ and $a\in\hH\times H.$ 
   This implies that the map $\varDelta_S$ is $\hH$-equivariant and condition (1) holds.  It is also directly seen from the definition that $\varDelta_N(\theta)=\varphi.$
   
   Now we are going to prove (4). Let $\chi\in\IBr(S\which\theta^{\hH}),$ and $V_1$ be a vertex of $\chi.$ Suppose that $\theta_1$ is an irreducible constituent of $\chi_N,$ and $\chi_1\in\IBr(S_{\theta}\which \theta_1)$ is the Clifford correspondent of $\chi.$ Let $a\in\hH\times H$ be such that $S_{\theta}^a\in\{S_1,\cdots,S_l\}$ and $\chi_1^a\in\IBr(S_{\theta}^a\which \theta).$ 
   Recall that $\varDelta_S(\chi)
   =(\varDelta^1_{S_{\theta}^a}(\chi_1^a)^{a^{-1}})^{S\cap H}$ by our definition.
   Since Clifford correspondent characters have a common vertex, we may assume that $V_1\leqslant S_{\theta}$ and is a vertex of $\chi_1.$ Thus $V_1^a$ is a vertex of $\chi_1^a.$ 
   By Proposition \ref{prop:H-triVt}, there exists a vertex $T$ of $\varDelta^1_{S^a_{\theta}}(\chi_1^a)$ such that $TN=V_1^aN.$ Let $V_2=T^{a^{-1}}.$ 
   Then $V_2$ is a vertex of $\varDelta^1_{S^a_{\theta}}(\chi_1^a)^{a^{-1}},$ and hence also a vertex of $\varDelta_S(\chi)
   =(\varDelta^1_{S_{\theta}^a}(\chi_1^a)^{a^{-1}})^{S\cap H}$.
   By the conjugation action of $a^{-1}$ on the equation $TN=V_1^aN$ we have $V_2N=V_1N.$ 
   
   What we left is to prove (2), and the proof is a little complicated. We divide it into 3 steps, which are from {\itshape Step 4} to {\itshape Step 6.}
   Let $\bar{G}=G/N$ and $\bar{g}$ be the image of $g\in G$ in $\bar{G}.$
   
   {\itshape Step 4.} For any $S\in\mS(G_{\theta},N),$ the bijection $\varDelta_S:\IBr(S\which\theta^{\hH})
   \ra\IBr(S\cap H\which \varphi^{\hH})$ restricts to a bijection $\varDelta_S^1:\IBr(S\which\theta)
   \ra\IBr(S\cap H\which\varphi)$ induced by $(\Pj_S,\Pj'_{S\cap H}).$ 
   
Let $S\in\mS(G_{\theta},N)$ and $\chi\in\IBr(S\which\theta).$ Suppose that $a\in\hH\times H$ is such that $S^a\in\{S_1,\cdots,S_l\}$ and $\chi^a\in\IBr(S^a\which\theta).$ 
Then by definition, we have $\varDelta_S(\chi)=\varDelta^1_{S^a}(\chi^a)^{a^{-1}}.$
Since $\chi$ and $\chi^a$ both lie over $\theta$ we have $a\in(\hH\times H)_{\theta}=(\hH\times H)_{\varphi}.$
Noticing that $\varDelta^1_{S^a}(\chi^a)\in\IBr(S^a\cap H\which \varphi),$ we have $\varDelta_S(\chi)\in\IBr(S\cap H\which\varphi).$ 
Because $|\IBr(S\which\theta)|=|\IBr(S\cap H\which\varphi)|,$ we see that $\varDelta_S$ restricts to a bijection $\varDelta_S^1:\IBr(S\which\theta)
\ra\IBr(S\cap H\which\varphi).$
Now we prove that $\varDelta_S^1$ is induced by $(\Pj_S,\Pj'_{S\cap H}).$ 
Notice that this is our definition when $S\in\{S_1,\cdots,S_l\}.$ 
Now $S$ is any subgroup in $\mS(G_{\theta},N).$ 
Let $\PQ$ be an irreducible projective representation of $\bar{S}$ with factor set $\alpha^{-1}_{\bar{S}\times\bar{S}}$ such that $\PQ\otimes \Pj_S$ affords $\chi,$ where $\alpha$ is the factor set associated with $\Pj$ and $\alpha_{\bar{S}\times\bar{S}}$ is the restriction of $\alpha$ to $\bar{S}\times\bar{S}$ (recall that we can regard $\alpha$ as a factor set of $\bar{G}_{\theta}$).
We need to prove that $\PQ\otimes\Pj'_{S\cap H}$ affords $\varDelta_{S^a}^1(\chi^a)^{a^{-1}}.$
Notice that $(\PQ\otimes\Pj_S)^a$ affords $\chi^a,$ and 
$$(\PQ\otimes\Pj_S)^a=\PQ^a\otimes(\Pj_S)^a
\sim \PQ^a\otimes \mu_{a,S^a}\cdot \Pj_{S^a}
=\PQ^a\cdot\mu_{a,S^a}\otimes \Pj_{S^a}.$$
Here $\PQ^a\cdot \mu_{a,{S^a}}:\overline{S^a}\ra \GL_s(F)$
($s$ is the degree of $\PQ$) is defined by $\PQ^a\cdot \mu_{a,S^a}(\bar{h})
=\PQ(\bar{h}^{\bar{g}^{-1}})^{\sigma}
\cdot \mu_a(h)$ for $h\in S^a$  and $a=(\sigma, g).$ 
Since $\PQ^a\cdot \mu_{a,S^a}\otimes \Pj_{S^a}$ is a group representation, we have that $\PQ^a\cdot\mu_{a,S^a}$ is a projective representation of $\overline{S^a}$ with factor set $\alpha^{-1}_{\overline{S^a}\times\overline{S^a}}.$ 
Thus by definition $\PQ^a\cdot \mu_{a,S^a}\otimes\Pj'_{S^a\cap H}$ affords $\varDelta_{S^a}^1(\chi^a)$. 
On the other hand, 
$$\PQ^a\cdot \mu_{a,S^a}\otimes\Pj'_{S^a\cap H}
=\PQ^a\cdot \mu'_{a,S^a\cap H}\otimes\Pj'_{S^a\cap H}
=\PQ^a \otimes \mu'_{a,S^a\cap H}\cdot \Pj'_{S^a\cap H}
=\PQ^a \otimes (\Pj'_{S\cap H})^a
= (\PQ \otimes\Pj'_{S\cap H})^a.$$
Hence $\PQ\otimes\Pj'_{S\cap H}$ affords $\varDelta_{S^a}^1(\chi^a)^{a^{-1}}.$   

{\itshape Step 5.} If $S,T\in\mS(G_{\theta},N)$ are such that $T\leqslant S,$ then for any $\chi\in\IBr(S\which\theta^{\hH})$ we have $\varDelta_S(\chi)_{T\cap H}=\varDelta_T(\chi_{T}).$  

Since the maps $\varDelta_S,\varDelta_T$ are $\hH$-equivariant, we may assume that $\chi\in\IBr(S\which\theta).$ 
Thus we only need to prove that $\varDelta_S^1(\chi)_{T\cap H}=\varDelta_T^1(\chi_T).$
Notice that in {\itshape Step 4}, we have proved that the bijections $\varDelta^1_S$ and $\varDelta^1_T$ are induced by $(\Pj_S,\Pj'_{S\cap H})$ and $(\Pj_T,\Pj'_{T\cap H}),$ respectively.
Suppose that $\PQ$ is an irreducible projective representation of $\bar{S}$ with factor set $\alpha^{-1}_{\bar{S}\times\bar{S}}$ such that the group representation $\PQ\otimes \Pj_S$ affords $\chi.$ 
Thus the group representation $\PQ_{\bar{T}}\otimes\Pj_T$ affords $\chi_T.$
Let $\PQ_1,\cdots,\PQ_t$ be all the irreducible constituents of $\PQ_{\bar{T}},$ that is, $\PQ_{\bar{T}}$ is similar with a projective representation like the following:
$$\begin{pmatrix}
\PQ_1 & * & \cdots & * \\
\textbf{0} & \PQ_2  & \cdots & * \\
\vdots & \vdots & \ddots & \vdots \\
\textbf{0} &  \textbf{0} & \cdots & \PQ_t 
 \end{pmatrix},$$ 
where each $\PQ_i,1\leqslant i\leqslant t,$  is an irreducible projective representation of $\bar{T}$ with factor set $\alpha^{-1}_{\bar{T}\times\bar{T}}.$
Then $\PQ_1\otimes\Pj_T,\cdots,\PQ_t\otimes\Pj_T$ are exactly all the irreducible constituents of $\PQ_{\bar{T}}\otimes\Pj_T.$
For any $1\leqslant i\leqslant t,$ let $\zeta_i\in\IBr(T\which\theta)$ be the irreducible Brauer character afforded by $\PQ_i\otimes\Pj_T.$ Then we have $\chi_T=\zeta_1+\cdots+\zeta_t.$ 
In a similar way, $\PQ_1\otimes\Pj'_{T\cap H},\cdots,\PQ_t\otimes \Pj'_{T\cap H}$ are exactly all the irreducible constituents of $\PQ_{\bar{T}}\otimes\Pj'_{T\cap H}.$
 Since the bijection $\varDelta_T^1$ is induced by $(\Pj_T,\Pj'_{T'}),$ and noticing that $\PQ_{\bar{T}}\otimes\Pj'_{T\cap H}$ affords $\varDelta^1_S(\chi)_{T\cap H},$ we have $\varDelta^1_S(\chi)_{T'}
=\varDelta^1_T(\zeta_1)+\cdots+\varDelta^1_T(\zeta_t)
=\varDelta^1_T(\chi_T).$ This completes the proof of this step.

{\itshape Step 6.} For any $S,T\in\mS(G,N)$  such that $T\leqslant S,$ and any $\chi\in\IBr(S\which\theta^{\hH}),$ we have $\varDelta_S(\chi)_{T\cap H}=\varDelta_T(\chi_{T}).$ 

Let $\{g_1,\cdots,g_k\}$ be a complete set of representatives of left $S_{\theta}T$-cosets in $S,$ and we require that each $g_j,j=1,\cdots,k,$ belongs to $S\cap H.$
Thus $S=g_1S_{\theta}T\sqcup\cdots\sqcup g_kS_{\theta}T$ (a disjoint union).  
Since $S_{\theta}\zg S,$ we have $g_jS_{\theta}T=S_{\theta}g_jT.$
Notice that $\{g_1,\cdots,g_k\}$ is also a complete set of representatives of left $(S_{\theta}\cap H)(T\cap H)$-cosets in $S\cap H.$
This is because $(S_{\theta}\cap H)(T\cap H)=S_{\theta}T\cap H,$ otherwise, $(S_{\theta}\cap H)(T\cap H)\lneqq S_{\theta}T\cap H,$ and hence $$S_{\theta}T=N(S_{\theta}\cap H)N(T\cap H)=N(S_{\theta}\cap H)(T\cap H)\lneqq N(S_{\theta}T\cap H)=S_{\theta}T,$$ a contradiction.
Let $\theta_1\in\theta^{\hH}$ be an irreducible constituent of $\chi_N$ and $\psi\in\IBr(S_{\theta}\which\theta_1)$ be the Clifford correspondent of $\chi.$ 
Thus by definition we have $\varDelta_S(\chi)=\varDelta_{S_{\theta}}(\psi)^{S\cap H}.$
By Mackey formula we have 
$$\chi_T=(\psi^S)_T=\sum_{j}((\psi^{g_j})_{T_{\theta}})^T,$$ where $j$ ranges over $1,\cdots,k.$
Similarly we have
$$\varDelta_S(\chi)_{T\cap H}
=(\varDelta_{S_{\theta}}(\psi)^{S\cap H})_{T\cap H}
=\sum_{j}\big((\varDelta_{S_{\theta}}
 (\psi)^{g_j})_{T_{\theta}\cap H}\big)^{T\cap H}
=\sum_j\big(\varDelta_{S_{\theta}}
 (\psi^{g_j})_{T_{\theta}\cap H}\big)^{T\cap H},$$
where the last equation comes from condition (1) that has been proved.
On the other hand we have 
$$\varDelta_T(\chi_T)
=\varDelta_T\big(\sum_j((\psi^{g_j})_{T_{\theta}})^T\big)
=\sum_j\varDelta_T\big( ((\psi^{g_j})_{T_{\theta}})^T 
 \big)
=\sum_j \varDelta_{T_{\theta}}((\psi^{g_j})_{T_{\theta}})^{T\cap H},$$
where the last equation is  actually the definition of $\varDelta_T,$ this is because that the irreducible constituents of $(\psi^{g_j})_{T_{\theta}}$ are exactly the Clifford correspondents of irreducible constituents of $((\psi^{g_j})_{T_{\theta}})^T,$ and $\varDelta_T,\varDelta_{T_{\theta}}$ are $\mathbb{N}$-linear.
In {\itshape Step 5}, we have proved that $$\varDelta_{T_{\theta}}((\psi^{g_j})_{T_{\theta}})
=\varDelta_{S_{\theta}}(\psi^{g_j})_{T_{\theta}\cap H}.$$
Thus we have $\varDelta_T(\chi_T)=\varDelta_S(\chi)_{T\cap H}.$
\end{proof}

\section{The  cohomology class associated with an $\hH$-triple}

Let \(^*: \mO \ra F\) be the natural ring homomorphism.
Recall \(F\) is a finite field.  
Denote by \(\Fp\) the prime subfield of \(F\), i.e., the minimal subfield of \(F\) containing \(p\) elements. 
Given a finite group \(N\) and \(\theta \in \IBr(N)\), let \(\Fp[\theta]\) be the field extension of \(\Fp\) in \(F\) obtained by adjoining the values \(\{\theta(x)^* \which x \in N_{p'}\}\).  
Set \(E = \Fp[\theta]\). By Brauer's theorem (see \cite[Chapter I, Theorem 19.3]{Feit}), there exists a representation \(X: N \ra \GL_m(F)\) affording \(\theta\) such that \(X(n) \in \GL_m(E)\) for all \(n \in N\), where \(m = \theta(1)\).  
In this case, we say \(X\) is \emph{realized} in \(E\) and may write \(X: N \ra \GL_m(E)\).  
If \(X': N \ra \GL_m(F)\) is another representation affording \(\theta\) and realized in \(E\), then by the Skolem-Noether theorem (see \cite[Chapter II, Theorem 4.6]{NagTsu}), there exists \(T \in \GL_m(E)\) such that \(X' = TXT^{-1}\).

In Section 2, we introduced $G$-algebras over $F$. This notion naturally extends to arbitrary fields. Given an $\hH$-triple $(G,N,\theta)_{\hH}$, we associate to it a $G$-algebra over $\Fp$, constructed as follows.

Let $(G,N,\theta)_{\hH}$ be an $\hH$-triple with $E = \Fp[\theta]$, and let $m = \theta(1)$. 
Fix a representation $X: N \ra \GL_m(E)$ affording $\theta$ and realized in $E$.
Denote by $\Gal(E/\Fp)$ the Galois group of the field extension $\Fp \subseteq E$; note that $|\Gal(E/\Fp)| = [E:\Fp]$ by Galois theory.
The following lemma holds.

\begin{lem}\label{lem:repAlg1}
Keep the notation above.
  For any $g\in G,$ there exists a unique $\sigma_g\in\Gal(E/\Fp)$ such that the group representation 
  $$X^{\sigma_gg}:N\ra\GL_m(E),n\mapsto X(gng^{-1})^{\sigma_g}$$ affords $\theta$.
  Moreover, $\sigma_g$ coincides with the restriction to $E$ of any $\tau \in \hH$ satisfying $\theta^\tau = \theta^{g^{-1}}$.
\end{lem}
\begin{proof}
  Let $\tau\in\hH$ be such that $\theta^{\tau}=\theta^{g^{-1}}.$ 
  We can regard $\tau$ as an automorphism of $F$ and we let $\sigma_g$ be the restriction of $\tau$ to $E.$
  Since the group representation $X^{\tau g}:N\ra\GL_m(F),n\mapsto X(gng^{-1})^{\tau}$ affords $\theta$ and $X^{\sigma_g g}(n)=X^{\tau g}(n)$ for any $n\in N,$ we see that $\sigma_g\in\Gal(E/\Fp)$ satisfies the condition in the lemma.
  Now we prove the uniqueness of $\sigma_g$ in $\Gal(E/\Fp).$ 
  The map $\theta^*:N\ra F,n\mapsto \theta(n_{p'})^*,$ where $n_{p'}$ is the $p'$-part of $n,$ is the $F$-character of group representation $X^{\sigma_g g}:N\ra \GL_m(F)$. 
  Notice that the image of $\theta^*$ is contained in $E.$
Recall that two (absolutely) irreducible representations of a group are similar if and only if they afford the same $F$-character (see \cite[Chapter II, Theorem 3.13]{NagTsu}). 
Let $\upsilon\in\Gal(E/\Fp)$ be such that $X^{\sigma_g \upsilon g}$ affords $\theta.$ Then $X^{\sigma_g g}$ and $X^{\sigma_g \upsilon g}$ have a common $F$-character that must be $\theta^*.$ 
  On the other hand, the $F$-character of $X^{\sigma_g \upsilon g}$ is $(\theta^*)^{\upsilon}:N\ra F,n\mapsto \theta^*(n)^{\upsilon}$ as $\theta^*$ is the $F$-character of $X^{\sigma_g g}.$ 
  This forces $\upsilon$ to stabilize every element $\theta^*(n)$ for $n\in N,$ which means $\upsilon=1.$ 
\end{proof}

In the above lemma, the element $\sigma_g$ is independent of the choice of $X$. We say that $\sigma_g \in \Gal(E/\Fp)$ is determined by the equation 
$\theta^{\sigma_g} = \theta^{g^{-1}}.$ 
Note that $\sigma_g = 1$ when $g \in G_\theta$.
For arbitrary $g \in G$, since $X^{\sigma_g g}$ is similar to $X$, there exists $T_g \in \GL_m(E)$ such that
$$X^{\sigma_g g}(n) = T_g X(n) T_g^{-1} \quad \text{for all } n \in N,$$
where $T_g$ is unique up to scalar multiples in $E^\times$.
Let $A = \M_m(E)$, identifying $\GL_m(E)$ with $A^\times$. For $x \in A$, define
$$x^g := T_g^{-1}\cdot  x^{\sigma_g}\cdot T_g,$$
which is independent of the choice of $T_g$. Then $x \mapsto x^g$ defines an $\Fp$-algebra automorphism of $A$. We have the following lemma.

\begin{lem}\label{lem:repAlg2}
Maintaining the above notation, the map 
$$A \times G \to A, \quad (x,g) \mapsto x^g$$
defines a group action of $G$ on $A$ by $\Fp$-algebra automorphisms. 
\end{lem}
\begin{proof}
  We need to prove that for any $g,h\in G$ and $x\in A$ we have $(x^g)^h=x^{gh}.$ 
  By Lemma \ref{lem:repAlg1}, we have $\sigma_{gh}=\sigma_g\sigma_h.$
  Let $T_g,T_h\in A^{\times}$ be such that $X^{\sigma_g g}=T_g X T_g^{-1},X^{\sigma_h h}=T_h X T_h^{-1},$ respectively.
  Then for any $n\in N,$ we have
  \begin{align*}
    X^{\sigma_{gh}(gh)}(n)&
       =X(\zs{gh}n)^{\sigma_{gh}}
       =X(\zs{g}(\zs{h}n))^{\sigma_g\sigma_h}
       =\big( T_g\cdot X(\zs{h}n)\cdot T_g^{-1}\big)^{\sigma_h}  \\
    &=T_g^{\sigma_h}\cdot X(\zs{h}n)^{\sigma_h} \cdot (T_g^{\sigma_h})^{-1}
    =T_g^{\sigma_h}\cdot T_h\cdot X(n)\cdot T_h^{-1}\cdot (T_g^{\sigma_h})^{-1}.
  \end{align*}
  Thus we can let $T_{gh}=T_g^{\sigma_h}T_h.$ Hence for any $x\in A$ we have 
  $$(x^g)^h=(T_g^{-1}\cdot x^{\sigma_g}\cdot T_g)^h
  =T_h^{-1}\cdot (T_g^{\sigma_h})^{-1}\cdot x^{\sigma_g\sigma_h}\cdot T_g^{\sigma_h}\cdot T_h=T_{gh}^{-1}\cdot x^{\sigma_{gh}}\cdot T_{gh}=x^{gh}.$$
  This completes the proof of the lemma.
\end{proof}

With this structure, $A$ becomes a $G$-algebra over $\Fp,$ and $X:N\ra A^{\times}=\GL_m(E)$ is a group representation of $N$ affording $\theta$ and realised in $E.$
We say $A$ is the \emph{$G$-algebra associated with} the $\hH$-triple $(G,N,\theta)_{\hH}$ (with respect to $X$). 
In fact, we have
\begin{equation}\label{equ:repalg1}
  X(n)^g=X(n^g)\quad \text{for all } n\in N \text{ and }
g\in G.
\end{equation}
This can be deduced by
$X(n)^{\sigma_g}= X^{\sigma_g g}(n^g)=T_gX(n^g)T_g^{-1}.$
Observe that $A$ is naturally an $E$-algebra, and the action of $g$ on $A$ is $\sigma_g$-linear, i.e.,
\begin{equation}\label{equ:repalg2}
  (z\cdot x)^g=z^{\sigma_g}\cdot x^g \quad \text{for any } z\in E \text{ and } x\in A.
\end{equation} 
Since $A$ is $E$-spanned by the set $\{X(n)\which n\in N\},$ the action of $g$ on $A$ is completely determined by (\ref{equ:repalg1}) and (\ref{equ:repalg2}).

We now recall basic notions of second cohomology.
Let $E$ be a field with a finite group $G$ acting as automorphisms, where $x^g$ denotes the action of $g \in G$ on $x \in E$.
A factor set of $G$ (with coefficients in $E^{\times}$) is a map 
$\alpha:G\times G\ra E^{\times}$ satisfying the cocycle condition:
$$\alpha(gh,k)\alpha(g,h)^k = \alpha(g,hk)\alpha(h,k) \quad \text{for all } g,h,k \in G.$$
The factor sets arising from projective representations occur when $G$ acts trivially on $E$.
If $\alpha(1,g) = \alpha(g,1) = 1$ for all $g \in G$, we say $\alpha$ is \emph{normalized}. Denote by $\Z^2(G,E^\times)$ the group of all factor sets of $G$.
Two factor sets $\alpha, \beta \in \Z^2(G,E^\times)$ are \emph{cohomologous} if there exists $\gamma \colon G \to E^\times$ with
$$\beta(g,h) = \alpha(g,h)\gamma(g)^h \gamma(h)\gamma(gh)^{-1} \quad \text{for all } g,h \in G.$$
The \emph{second cohomology group} $\bH^2(G,E^\times)$ consists of cohomology classes $[\alpha]$ of all factor sets $\alpha \in \Z^2(G,E^\times)$ under this relation.

We will associate to each $\hH$-triple a cohomology class of the quotient group.
 First, we establish the following lemma.
 Note that a unitary embedding of rings is an injective ring homomorphism that preserves identity elements.

\begin{lem}\label{lem:coho1}
  Let $E$ be any field and $K\subseteq E$ be a subfield such that $s=|E:K|$ is finite. 
  Let $A=\M_m(E)$ for some positive integer $m.$ 
  Then there exists a unitary embedding $\iota:A\ra \M_{ms}(K)$ of $K$-algebras. And if $\iota':A\ra\M_{ms}(K)$ is another such embedding,  there exists $Y\in\M_{ms}(K)^{\times}$ satisfying $\iota'(x)=Y^{-1}\iota(x)Y$ for all $x\in A.$
\end{lem}
\begin{proof}
  Let $V$ be a simple right $A$-module. A canonical example is $V = E^m$, the space of row vectors over $E$ of dimension $m$, where $A = \M_m(E)$ acts by right multiplication.   
  Let $\fE=\End_K(V)$ be the $K$-endomorphism algebra, where endomorphisms are composed from the left. Note that $\fE$ is a $K$-algebra.
  Fixing a $K$-basis for $V$, we identify $V \cong K^{ms}$ as $K$-vector spaces (since $\dim_K V = ms$), and $\fE \cong \M_{ms}(K)$ acting on $V$ by right multiplication.
The right $A$-action on $V$ induces an injective $K$-algebra homomorphism
$$\iota \colon A \hookrightarrow \fE, \quad x \mapsto (v \mapsto vx),$$
which preserves identities by construction.

 Let $\iota,\iota':A\ra \M_{ms}(K)$ be two unitary embeddings.
 Let $V = K^{ms}$ be the row vector space over $K.$
 For any $v\in V$ and $x\in A,$ define $v\cdot x=v\iota(x)$ and $v\ast x=v\iota'(x).$ 
  Thus $V$ has two $A$-module structures, and we denote this two $A$-modules by $V(\cdot)$ and $V(\ast),$ respectively.  
  Since $A$ is simple, $V(\cdot) \cong V(\ast)$ as $A$-modules. This isomorphism is realized by a $K$-linear bijection $f \colon V \to V$ satisfying:
  \begin{equation}\label{equ:new1}
    f(v \cdot x) = f(v) \ast x \quad \forall v \in V, x \in A.
  \end{equation}
The $K$-linearity implies $f(v) = vY$ for some invertible $Y \in \GL_{ms}(K)$. 
Substituting into \ref{equ:new1} yields:
$$v\iota(x)Y = vY\iota'(x) \quad \forall v \in V, x\in A.$$
 Since this holds for all $v \in V$, we conclude that $\iota(x)Y = Y\iota'(x),$ hence $\iota'(x) = Y^{-1}\iota(x)Y.$
\end{proof}

Observe that in the above lemma we have
$$\C_{\M_{ms}(K)}(\iota(A)) = \iota(E\cdot 1_A).$$
This equality holds because
$\C_{\M_{ms}(K)}(\iota(A)) \supseteq \iota(E\cdot 1_A)$ by construction
     and $$\dim_K \C_{\M_{ms}(K)}(\iota(A)) = s$$ by \cite[Chapter II, Lemma 4.4]{NagTsu}.
In fact, $\iota(E\cdot 1_A)$ and $\iota(A)$ are both centralizers of each other in $\M_{ms}(K)$. 
Consequently, the invertible matrix $Y$ in the lemma is unique up to left multiplication by elements of $\iota(E^\times \cdot 1_A)$ (or right multiplication by elements of $\iota'(E^\times \cdot 1_A)$).

The next result extends the theory of projective extensions for character triples to the broader framework of $\hH$-triples.

\begin{thm}\label{thm:coho}
  Let $(G,N,\theta)_{\hH}$ be an $\hH$-triple with $E=\Fp[\theta].$
  Let $X:N\ra \GL_m(E)=A^{\times}$ be a group representation affording $\theta$ and realised in $E,$ where $m=\theta(1).$
  Let $A$ be the $G$-algebra associated with $(G,N,\theta)_{\hH}$ and $X,$ and let $\iota:A\ra\M_{ms}(\Fp)$ be a unitary embedding, where $s=|E:\Fp|$.
Then there exists a function $Y:G\ra\M_{ms}(\Fp)^{\times}$ satisfying:
  \begin{enumerate}
    \item  $Y(g)^{-1}\iota(x)Y(g)=\iota(x^g)$ for all $x \in A$,  $g \in G$.
    \item  $Y(n)=\iota(X(n))$ for all $n\in N.$
    \item $Y(gn)=Y(g)Y(n)$ and $Y(ng)=Y(n)Y(g)$ for all $n \in N$,  $g \in G$.
  \end{enumerate}
\end{thm}
\begin{proof}
By Lemma \ref{lem:coho1}, there exists a unitary embedding $\iota \colon A \hookrightarrow \M_{ms}(\Fp)$. 
  Let $\{g_1=1,\ldots,g_t\}$ be a complete set of coset representatives for $N$ in $G$.
For each $1 \leqslant i \leqslant t$, the function
$x \mapsto \iota(x^{g_i})$
defines another unitary $\Fp$-algebra embedding. Lemma \ref{lem:coho1} guarantees the existence of matrices $Y(g_i) \in \GL_{ms}(\Fp)$ satisfying:
$$Y(g_i)^{-1}\iota(x)Y(g_i) = \iota(x^{g_i}) \quad \text{for all } x \in A,$$
where we set $Y(1) = \I_{ms}$.
For arbitrary $g \in G$, writing $g = ng_i$ uniquely with $n \in N$ and $1 \leqslant i \leqslant t$, we define
$$Y(g) := \iota(X(n))Y(g_i).$$
Standard verification (cf. the proof for \cite[Chapter III, Theorem 5.7]{NagTsu}) shows that $Y$ satisfies all required properties.
\end{proof}

In the above theorem, for any $g,h \in G$, observe that 
$$Y(gh)^{-1}Y(g)Y(h) \in \iota(E^\times \cdot 1_A),$$
since this element centralizes $\iota(A)$. We may therefore define $\alpha(g,h) \in E^\times$ via the relation
$$Y(g)Y(h) = Y(gh)\iota(\alpha(g,h)).$$

\begin{prop}
  Keep the notation above. Let $G$ act on the field $E$ through automorphisms $\{\sigma_g\which g\in G\}$, where each $\sigma_g\in\Gal(E/\Fp)$ is determined by  $\theta^{\sigma_g}=\theta^{g^{-1}}.$ 
   Explicitly, this action is given by:
$$x^g := x^{\sigma_g} \quad \text{for all } x \in E, \, g \in G.$$
  Then the function $\alpha \colon G \times G \to E^\times$ defined previously constitutes a factor set in $\Z^2(G,E^\times)$.
\end{prop}
\begin{proof}
  For any $g,h,k\in G.$ On the one hand, we have
  $$Y(g)(Y(h)Y(k))=Y(g)Y(hk)\iota(\alpha(h,k))
  =Y(ghk)\iota(\alpha(g,hk)\alpha(h,k)).$$ 
  On the other hand, we have
  $$(Y(g)Y(h))Y(k)=Y(gh)\iota(\alpha(g,h))Y(k)
  =Y(gh)Y(k)\iota(\alpha(g,h)^k)
  =Y(ghk)\iota(\alpha(gh,k)\alpha(g,h)^k).$$
  This implies that $\alpha(g,hk)\alpha(h,k)=\alpha(gh,k)\alpha(g,h)^k,$ thus $\alpha\in \Z^2(G,E^{\times}).$
\end{proof}

Let $\bar{G} = G/N$ and write $\bar{g} \in \bar{G}$ for the image of $g \in G$. 
Since $N$ acts trivially on $E$, the $G$-action on $E$ descends to a
 $\bar{G}$-action on $E$ by $x^{\bar{g}}=x^g$ for $x\in E$ and $g\in G.$ 
For any $g,h\in G,$ observe that $\alpha(g,h)$ depends only on the $N$-cosets of $g$ and $h$. 
Therefore, the induced function
$$\bar{\alpha} \colon \bar{G} \times \bar{G} \to E^\times, \quad (\bar{g}, \bar{h}) \mapsto \alpha(g,h)$$
is a well-defined factor set in $\Z^2(\bar{G}, E^\times)$.

\begin{prop}
  Keep the notation above. The cohomology class $[\bar{\alpha}]\in\tH^2(\bar{G},E^{\times})$ is uniquely determined by the $\hH$-triple $(G,N,\theta)_{\hH}.$
\end{prop}
\begin{proof}
Given an $\hH$-triple $(G,N,\theta)_{\hH}$, the cohomology class $[\bar{\alpha}]$ is constructed through the following steps.
\begin{enumerate}
    \item First, choose a representation $X$ affording $\theta$ and realized in $E:=\Fp[\theta],$ and let $A$ be the $G$-algebra associated with $(G,N,\theta)_{\hH}$ and $X$.
    \item Next, select a unitary $\Fp$-algebra embedding $\iota \colon A \hookrightarrow \M_{ms}(\Fp)$.
    \item Finally, fix a function $Y \colon G \to \GL_{ms}(\Fp)$ satisfying the conditions of Theorem~\ref{thm:coho}.
\end{enumerate} 
\item Then $\bar{\alpha} \colon \bar{G} \times \bar{G} \to E^\times$ is defined via the relation
    $$Y(g)Y(h) = Y(gh)\iota(\bar{\alpha}(\bar{g},\bar{h})) \quad \text{for } g,h \in G.
    $$

Let $X'$ be another representation affording $\theta$ and realized in $E,$ and let $A'$ be the $G$-algebra associated with $(G,N,\theta)_{\hH}$ and $X'$.
Let $\iota':A'\ra \M_{ms}(\Fp)$ be a unitary embedding, and $Y':G\ra \M_{ms}(\Fp)^{\times}$ be a function satisfying Theorem \ref{thm:coho}  for $A',X',\iota',Y'.$
Define $\bar{\alpha}' \colon \bar{G} \times \bar{G} \to E^\times$ via 
$$
Y'(g)Y'(h) = Y'(gh)\iota'(\bar{\alpha}'(\bar{g}, \bar{h})) \quad \forall g,h \in G.
$$ 
We need to prove that $[\bar\alpha]=[\bar\alpha'].$

Since both $X$ and $X'$ afford $\theta$, there exists $T \in \GL_m(E)$ such that
$$ X'(n) = TX(n)T^{-1} \quad \forall n \in N. $$
Let $\ast$ denote the $G$-action on $A'$ (distinguished from the $G$-action on $A$). 
Thus $X'(n)^{\ast g}=X'(n^g)$ for all $n\in N$ and $g\in G.$
The $E$-algebra isomorphism 
$$ f \colon A \to A', \quad x \mapsto TxT^{-1} $$
preserves $G$-actions.
First, for generators $X(n)$ ($n \in N$): $$f(X(n))^{\ast g}=X'(n)^{\ast g}=X'(n^g)=f(X(n^g))=f(X(n)^g).$$
For general $x = \sum_{n \in N} k_n X(n)$ ($k_n \in E$): 
$$f(x)^{\ast g}=\sum_{n\in N}k_n^gf(X(n))^{\ast g}
=\sum_{n\in N}k_n^gf(X(n)^g)
=f(\sum_{n\in N}k_n^gX(n)^g)
=f(x^g).$$
Thus $f:A\ra A'$ is an isomorphism of $G$-algebras, allowing us to assume that $A=A'$ and $X=X'.$

  Given two unitary embeddings $\iota, \iota' \colon A \to \M_{ms}(\Fp)$, Lemma~\ref{lem:coho1} provides $W \in \GL_{ms}(\Fp)$ such that
$$\iota(x) = W\iota'(x)W^{-1} \quad \forall x \in A.$$
Define $\tilde{Y}(g) = WY'(g)W^{-1}$ for $g \in G$. 
For any $g,h\in G,$ conjugating the relation
$$Y'(g)Y'(h) = Y'(gh)\iota'(\bar{\alpha}'(\bar{g},\bar{h}))$$
by $W$ yields:
\begin{equation}\label{equ:coho-Y}
\tilde{Y}(g)\tilde{Y}(h) = \tilde{Y}(gh)\iota(\bar{\alpha}'(\bar{g},\bar{h})).
\end{equation}
For any $x\in A$ and $g\in G,$ conjugating $Y'(g)^{-1}\iota'(x)Y'(g) = \iota'(x^g)$ by $W$ gives
$$\tilde{Y}(g)^{-1}\iota(x)\tilde{Y}(g) = \iota(x^g) \quad \forall x \in A, g \in G.$$
Thus for any $g\in G,$ there exists $\gamma(g) \in E^\times$ with $\tilde{Y}(g) = Y(g)\iota(\gamma(g))$. 
Since $\gamma(g)$ depends only on the coset $gN$, we define $\bar{\gamma} \colon \bar{G} \to E^\times$ by $\bar{\gamma}(\bar{g}) = \gamma(g)$.
For any $g,h\in G,$ we compute that 
\begin{equation}\label{equ:coho-Y2}
  \tilde Y(g)\tilde Y(h)
  =Y(g)\iota(\bar \gamma (\bar g))
  Y(h)\iota(\bar \gamma (\bar h))
  =\tilde Y(gh)\iota(\bar\alpha(\bar g,\bar h)\bar\gamma (\bar g)^{\bar h}
  \bar\gamma(\bar h)\bar\gamma (\bar{g}\bar{h})^{-1}).
\end{equation}
Comparing \ref{equ:coho-Y} with \ref{equ:coho-Y2}, 
we obtian $$\bar{\alpha}'(\bar g,\bar h)=\bar\alpha(\bar g,\bar h)\bar\gamma (\bar g)^{\bar h}
  \bar\gamma(\bar h)\bar\gamma (\bar{g}\bar{h})^{-1}\quad \forall g,h
  \in G.$$ 
Hence $[\bar{\alpha}'] = [\bar{\alpha}]$ in $\bH^2(\bar{G},E^\times)$.
\end{proof}

In the above proposition, the cohomology class $[\bar\alpha] \in \bH^2(\bar{G},E^{\times})$ is called the \emph{cohomology class associated with} the $\hH$-triple $(G,N,\theta)_{\hH}$, or equivalently, we say that $(G,N,\theta)_{\hH}$ \emph{realizes} the cohomology class $[\bar\alpha]$.

In the remainder of this section, we will prove that when two $\hH$-triples share the same cohomology class, they satisfy the partial order relation from Definition~\ref{def:H-tri:ord}. We begin with a key lemma.

\begin{lem}\label{lem:coho->H-tri}
Let $(G,N,\theta)_{\hH}$ be an $\hH$-triple with $E = \Fp[\theta]$, and let $[\bar\alpha] \in \bH^2(G/N, E^\times)$ be its associated cohomology class. 
  Let $\alpha:G\times G\ra E^{\times}$ be a factor set that is a lift of some $\bar\alpha\in[\bar\alpha]$.
  Then there exists a projective representation $\Pj$ of $G_{\theta}$ associated with $\theta$ with factor set $\alpha_{G_{\theta}\times G_{\theta}},$ and for any $a=(\sigma, g)\in(\hH\times G)_{\theta},$ the function $\mu_a:G_{\theta}\ra F^{\times}$ determined by $\Pj^a\sim \mu_a\Pj$ satisfies
  $$\mu_a(h)=\big(\alpha(g,hg^{-1})^{-1}
  \alpha(h,g^{-1})^{-1} \alpha(g,g^{-1}) \big)^{\sigma}
  \quad \text{for all } h\in G_{\theta}.$$ 
  
\end{lem}
\begin{proof}
Let $X:N\ra\GL_{m}(E)$ (where $m=\theta(1)$) be a representation affording $\theta$ and realized in $E,$ and let $A$ be the $G$-algebra associated with $(G,N,\theta)_{\hH}$ and $X$.
Fix a unitary embedding $\iota \colon A \ra \M_{ms}(\Fp).$ 
Let $Y:G\ra \GL_{ms}(\Fp)$ be a function satisfying the conditions of Theorem \ref{thm:coho}, and $$Y(g)Y(h)=Y(gh)\iota(\alpha(g,h))\quad\forall g,h\in G.$$

  For any $h\in G_{\theta},$ since $Y(h)$ centralizes $\iota(E\cdot 1_A)$ and the centralizer of $\iota(E\cdot 1_A)$ in $\M_{ms}(\Fp)$ is $\iota(A)$, we have $Y(h)\in \iota(A).$ 
  Let $\Pj(h)\in A^{\times}$ be the unique element such that $\iota(\Pj(h))=Y(h).$
  Thus the function $$\Pj:G_{\theta}\ra \GL_m(F),h\mapsto \Pj(h),$$ is a projective representation associated with $\theta$ with factor set $\alpha_{G_{\theta}\times G_{\theta}}.$ 
  (Note that $A^\times = \GL_m(E) \subseteq \GL_m(F)$ by construction.)

  Let $a=(\sigma, g)\in(\hH\times G)_{\theta}.$ 
  For any  $h\in G_{\theta},$  since $Y(g)Y(h)Y(g)^{-1}Y(ghg^{-1})^{-1}$ centralizes $\iota(A),$ there exists $\Psi(g,h)\in E^{\times}$ such that 
  \begin{equation}\label{equ:new2}
    Y(g)Y(h)Y(g)^{-1}=\iota(\Psi(g,h))Y(ghg^{-1}).
  \end{equation}
  Let us compute $\Psi(g,h)$ first. 
  As $Y(g)Y(g^{-1})=Y(1)\iota(\alpha(g,g^{-1})),$ we have $$Y(g)^{-1}=Y(g^{-1})\iota(\alpha(g,g^{-1})^{-1}).$$
  Hence
  \begin{align*}
    Y(g)Y(h)Y(g)^{-1}&
    =Y(g)Y(h)Y(g^{-1})\iota(\alpha(g,g^{-1})^{-1})  \\
     &=Y(ghg^{-1})\iota(\alpha(g,hg^{-1})
  \alpha(h,g^{-1})\alpha(g,g^{-1})^{-1}).
  \end{align*} 
  Thus  $$\Psi(g,h)=\alpha(g,hg^{-1})
  \alpha(h,g^{-1})\alpha(g,g^{-1})^{-1}.$$
  From \ref{equ:new2},
we take the inverse image under $\iota$ to obtain
\begin{equation}\label{equ:new3}
  \Pj(h)^{g^{-1}} = \Psi(g,h)\Pj(ghg^{-1}),
\end{equation} where $\Pj(h)^{g^{-1}}$ is the action of $g^{-1}$ on $\Pj(h)\in A$ via the $G$-algebra structure. 
  Recall that 
\begin{equation}\label{equ:new4}
  \Pj(h)^{g^{-1}} =T_{g^{-1}}^{-1}\cdot \Pj(h)^{\sigma^{-1}}\cdot T_{g^{-1}},            
\end{equation} where $T_{g^{-1}}\in\GL_m(E)$ satisfies $X^{(\sigma, g)^{-1}}=T_{g^{-1}}XT_{g^{-1}}^{-1}$ 
(note that $\sigma_g=\sigma|_E$  by Lemma \ref{lem:repAlg1}). 
Define $T := (T_{g^{-1}}^{-1})^{\sigma}$. 
Combining \ref{equ:new3} and \ref{equ:new4}, we obtain the identity
\[
T_{g^{-1}}^{-1} \cdot \Pj(h)^{\sigma^{-1}} \cdot T_{g^{-1}} = \Psi(g,h) \Pj(ghg^{-1}).
\]
Conjugating both sides by $\sigma$ yields
\[
  T \Pj(h) T^{-1} = \Psi(g,h)^{\sigma} \cdot \Pj(ghg^{-1})^{\sigma},
\]
or equivalently,  $$\Pj^a(h)=(\Psi(g,h)^{\sigma})^{-1}\cdot T\Pj(h)T^{-1}.$$
Thus $\mu_a(h)=(\Psi(g,h)^{\sigma})^{-1}
  =\big(\alpha(g,hg^{-1})^{-1}
  \alpha(h,g^{-1})^{-1} \alpha(g,g^{-1}) \big)^{\sigma},$ completing the proof.
\end{proof}

\begin{thm}\label{thm:coho->H-tri}
  Suppose that $(G,N,\theta)_{\hH}$ and $(H,C,\varphi)_{\hH}$ are $\hH$-triples such that 
\begin{enumerate}
  \item $G=NH$ and $N\cap H=C.$
  \item $(\hH\times H)_{\theta}=(\hH\times H)_{\varphi}.$
\end{enumerate}
Then the following hold.
\begin{enumerate}
  \item $\Fp[\theta]=\Fp[\varphi].$
  \item Let $E=\Fp[\theta],$ and let $\bar{G}=G/N=H/C.$ 
Then the two $\bar{G}$-actions on $E$ induced by the $\hH$-triples $(G,N,\theta)_{\hH}$ and $(H,C,\varphi)_{\hH}$ coincide.
  \item Let $[\bar\alpha],[\bar\beta]\in \tH^2(\bar{G},E^{\times})$ be the cohomology class associated with $(G,N,\theta)_{\hH}$ and $(H,C,\varphi)_{\hH},$ respectively.
If $[\bar\alpha]=[\bar\beta],$ then $$(G,N,\theta)_{\hH}\geqslant (H,C,\varphi)_{\hH}.$$
\end{enumerate}
 \end{thm}
\begin{proof}
  We prove $\Fp[\theta]=\Fp[\varphi]$ first.
  Recall that we can regard elements in $\hH$ as automorphisms of $F.$
  Let $\hH_0$ be the subgroup of $\hH$ consisting of automorphisms acting trivially on $F.$
  Then $\hH_0\zg\hH$ and $\hH/\hH_0$ is the Galois group of the field extension $\Fp\subseteq F$ ($\Fp\subseteq F$ is a Galois extension because $F$ is finite).
  Since $\hH_{\theta}/\hH_0,\hH_{\varphi}/\hH_0\leqslant \hH/\hH_0$ are subgroups corresponding to the intermediate fields $\Fp[\theta]$ and $\Fp[\varphi],$ respectively, and  $\hH_{\theta}=\hH_{\varphi}$ by our assumption, we have $\Fp[\theta]=\Fp[\varphi]$ by the fundamental theorem of Galois theory.
  
   For any $g\in H,$ let $\sigma_g,\tau_g\in\Gal(E/\Fp)$ be determined by the equations $\theta^{\sigma_g}=\theta^{g^{-1}}$ and $\varphi^{\tau_g}=\varphi^{g^{-1}},$ respectively. 
By Lemma \ref{lem:repAlg1}, each of $\sigma_g$ and $\tau_g$ is the restriction to $E$ of any $\sigma\in \hH$ such that $(\sigma, g)\in(\hH\times H)_{\theta}=(\hH\times H)_{\varphi}$.
Thus $\sigma_g=\tau_g,$ proving (2).
  
  Now assume that $[\bar\alpha]=[\bar\beta].$
  Let $\alpha:G\times G\ra F^{\times}$ and $\beta:H\times H\ra F^{\times}$ be two factor sets that are liftings of some $\bar\alpha\in[\bar\alpha]$ and $\bar\beta\in[\bar\beta],$ respectively, such that $\beta=\alpha_{H\times H}.$ 
  By Lemma \ref{lem:coho->H-tri}, there are projective representations $\Pj$ of $G_{\theta}$ and $\Pj'$ of $H_{\varphi}$ associated with $\theta$ and $\varphi$ with factor sets $\alpha$ and $\beta,$ respectively, and that for any $a=(\sigma, g)\in(\hH\times H)_{\theta},$ the functions $\mu_a:G_{\theta}\ra F^{\times}$ and $\mu'_a:H_{\varphi}\ra F^{\times}$ determined by $\Pj^a\sim\mu_a\Pj$ and $\Pj'^a\sim\mu'_a\Pj'$ satisfy 
  \begin{align*}
    \mu_a(h) &=\big(\alpha(g,hg^{-1})^{-1}
  \alpha(h,g^{-1})^{-1} \alpha(g,g^{-1}) \big)^{\sigma},\quad \forall h\in H_{\theta},\quad \text{and}  \\
    \mu'_a(h)&=\big(\beta(g,hg^{-1})^{-1}
  \beta(h,g^{-1})^{-1} \beta(g,g^{-1}) \big)^{\sigma},\quad
  \forall h\in H_{\theta}. 
  \end{align*}   
  Thus $\mu_a$ and $\mu'_a$ agree on $H_{\theta}.$ 
  Summarizing above, we see $(\Pj,\Pj')$ gives $$(G,N,\theta)_{\hH}\geqslant (H,C,\varphi)_{\hH}.$$ 
  This completes the proof.
\end{proof}
\noindent{\it Remark.} The converse of the above theorem does not hold. Specifically, the partial order relation $(G,N,\theta)_{\hH}\geqslant (H,C,\varphi)_{\hH}$ does not imply that their associated cohomology classes coincide.

\section{The final proof of Theorem A}
The group relations in Theorem~A, namely $G = HN$ and $C = H \cap N$, 
follow directly from the assumptions. Moreover, the identity
\[
(\hH \times H)_\theta = (\hH \times H)_\varphi
\]
holds due to the uniqueness properties of the DGN correspondence: 
$\varphi$ is the unique irreducible complex character of $C$ with $p$-defect zero satisfying $p \nmid (\varphi^N, \theta)_N$, 
and $\theta$ is the unique $M$-invariant irreducible complex character of $N$ with $p$-defect zero satisfying $p \nmid (\varphi^N, \theta)_N$ 
(see \cite[Chapter~V, Theorem~12.1]{NagTsu}).

Since $(G, N, \theta)_{\hH}$ is an $\hH$-triple by assumption, the equality $(\hH \times H)_\theta = (\hH \times H)_\varphi$ implies that $(H, C, \varphi)_{\hH}$ is also an $\hH$-triple. To prove that $(G, N, \theta)_{\hH}$ and $(H, M, \varphi)_{\hH}$ share the same cohomology class, we apply Dade's fusion splitting theorem, introduced below.


\subsection{Dade's fusion splitting theorem}
Several versions of Dade's fusion splitting theorem exist in the literature. In this work, we employ Turull's formulation from \cite[Theorem 4.5]{Tu08}. 
While Turull's original treatment was restricted to finite abelian $p$-groups, we extend these results to arbitrary finite $p$-groups through an inductive approach.

Let $D$ be a finite $p$-group. We begin by recalling the notion of Dade $D$-algebras. 
Let $E$ be a finite field of characteristic $p$, and let $A$ be a $D$-algebra over $E$.
We say $A$ is a \emph{Dade $D$-algebra} (over $E$) if $A$ is a full matrix algebra over $E$, and admits a $D$-stable $E$-basis with at least one trivial $D$-orbit.

For any Dade $D$-algebra $A$, there exists a unique group homomorphism $\rho:D\to A^{\times}$ satisfying
\[
x^d = \rho(d)^{-1}x\rho(d) \quad \text{for all } d\in D \text{ and } x\in A.
\]
We denote by $D_1 := \rho(D)$ the image of $D$ in $A^{\times}$.

The \emph{Brauer quotient} of $A$ is defined as
\[
\bar{A}(D) := A^D\bigg/\sum_{Q<D} A^D_Q,
\]
with the canonical \emph{Brauer homomorphism} $\br^A_D: A^D \to \bar{A}(D)$ (see \cite[\S 11]{Thevenaz} for details). 
By definition of the Dade $D$-algebra, we always have $\bar{A}(D) \neq 0$. 
When the algebra $A$ is clear from context, we simply write $\br_D$ for $\br^A_D$.

\begin{lem}\label{lem:Dade}
  Let $D$ be a finite $p$-group and let $A$ be a Dade $D$-algebra over a finite field $E$ of characteristic $p$. Then the Brauer quotient $\bar{A}(D)$ is a full matrix algebra over $E$.
\end{lem}
\begin{proof}

We argue by induction on $|D|.$
If $D$ is abelian, the result follows immediately from \cite[Proposition 3.8]{Tu08}.
Otherwise let $Z=\Z(D)$.
Since $A$ is a Dade $D$-algebra, it is in particular a Dade $Z$-algebra.
Since $Z$ is abelian, we have $\bar{A}(Z)$ is a full matrix algebra over $E.$
Moreover, the $D$-action on $A$ naturally induces a $D/Z$-action on $\bar{A}(Z)$, making $\bar{A}(Z)$ a Dade $D/Z$-algebra. Since $|D/Z| < |D|$, the inductive hypothesis applies, showing that the second Brauer quotient $\overline{\bar{A}(Z)}(D/Z)$ is again a full matrix algebra over $E$.

We are left to prove that 
\begin{equation}\label{equ:iso}
  \bar{A}(D) \cong \overline{\bar{A}(Z)}(D/Z).
\end{equation}
Let \( X \) be a \( D \)-stabilized \( E \)-basis of \( A \), and let \( X^D \) denote the set of \( D \)-fixed elements in \( X \). 
Consider the Brauer homomorphisms
\[
\operatorname{br}_Z \colon A^Z \to \bar{A}(Z) 
\quad \text{and} \quad 
\operatorname{br}_{D/Z} \colon \bar{A}(Z)^{D/Z} \to \overline{\bar{A}(Z)}(D/Z).
\]
We can compute from the definition that the set
\(
\left\{ \operatorname{br}_{D/Z} \circ \operatorname{br}_Z(x) \mid x \in X^D \right\}
\)
forms an \( E \)-basis of \( \overline{\bar{A}(Z)}(D/Z) \), and that
\(
\operatorname{br}_{D/Z} \circ \operatorname{br}_Z(\hat{x}) = 0
\)
for any \( x \in X \setminus X^D \), where \( \hat{x} \) denotes the sum over the \( D \)-orbit of \( x \).
Thus, the composition 
\[
\operatorname{br}_{D/Z} \circ \operatorname{br}_Z \colon A^D \to \overline{\bar{A}(Z)}(D/Z)
\]
is, in fact, the \( D \)-Brauer homomorphism of \( A \), thereby proving \eqref{equ:iso}.
\end{proof}

\begin{thm}\label{thm:Dade}
Let $A$ be a Dade $D$-algebra over $E$ and let $D_1$ be the image of $D$ in $A^{\times}$ as described above.
Let $B:=\bar{A}(D)$ be the Brauer quotient of $A$ with Brauer homomorphism $\mathrm{br}_D\colon A^D \to B$.
 Let $m=\sqrt{\dim_EA},l=\sqrt{\dim_EB}$ and $s=|E:\Fp|.$
 Fix unitary embeddings
          \begin{align*}
          \iota &\colon A \hookrightarrow \mathfrak{E}_1 := \mathrm{M}_{ms}(\mathbb{F}_p), \quad\text{and}\\
          \iota' &\colon B \hookrightarrow \mathfrak{E}_2 := \mathrm{M}_{ls}(\mathbb{F}_p).
          \end{align*}
 Then there exists a group homomorphism $$\phi:\N_{\fE_1^{\times}}(\iota(D_1))\ra \fE_2^{\times}$$ such that $$\phi(\iota(c))=\iota'(\br_D(c)),\quad \forall c\in (A^D)^{\times} .$$ 
\end{thm}
\begin{proof}
  With no loss of generality, we may assume the structural homomorphism 
$\rho\colon D\to D_1$ is faithful. We proceed by induction on $|D|$.

When $D$ is abelian, the result follows immediately from \cite[Theorem 4.5]{Tu08}.

Suppose $D$ is non-abelian.
Let $Z = \mathbf{Z}(D)$ be its center.
Since $\bar{A}(Z)$ is a full matrix algebra over $E$, we fix a unitary embedding
\[
\iota_Z \colon \bar{A}(Z) \hookrightarrow \mathfrak{E}_Z := \mathrm{M}_{ts}(\mathbb{F}_p),
\]
where $t = \sqrt{\dim_E \bar{A}(Z)}$.
As $Z$ is abelian, there exists a group homomorphism
\[
\phi_1 \colon \N_{\mathfrak{E}_1^\times}(\iota(Z_1)) \to \mathfrak{E}_Z^\times
\]
satisfying $\phi_1(\iota(c)) = \iota_Z(\mathrm{br}_Z(c))$ for all $c \in (A^Z)^\times$, where $Z_1 = \rho(Z)$ and $\mathrm{br}_Z \colon A^Z \to \bar{A}(Z)$ is the Brauer homomorphism.
Now observe that $\bar{A}(Z)$ becomes a Dade $D/Z$-algebra, and let $(D/Z)_1$ denote the image of $D/Z$ in $\bar{A}(Z)^\times$.
Since $(D/Z)_1=\br_Z(D_1),$ we have $\phi_1(\iota(D_1))=\iota_Z((D/Z)_1).$
From the proof of Lemma~\ref{lem:Dade}, we identify $\overline{\bar{A}(Z)}(D/Z)$ with $B$ via the established isomorphism.
Let $\br_{D/Z}:\bar{A}(Z)^{D/Z}\ra B$ be the Brauer homomorphism.
By induction, there exists a group homomorphism $$\phi_2:\N_{\fE_Z^{\times}}\big(\iota_Z((D/Z)_1)\big)\ra \fE_2^{\times}$$ such that $\phi_2(\iota_Z(x))=\iota'(\br_{D/Z}(x))$ for all $x\in(\bar A(Z)^{D/Z})^{\times}.$
Since $Z_1$ is characteristic in $D_1,$
  we have $\N_{\fE^{\times}_1}(\iota(D_1))\subseteq \N_{\fE^{\times}_1}(\iota(Z_1)).$
Also we have $$\phi_1\big(  \N_{\fE^{\times}_1}(\iota(D_1))  \big)  \subseteq  
  \N_{\fE^{\times}_Z}\big(\iota_Z((D/Z)_1)\big)$$ as $\phi_1(\iota(D_1))=\iota_Z((D/Z)_1).$
  Thus we can define the group homomorphism
  $$\phi:\N_{\fE^{\times}_1}(\iota(D_1)) \ra \fE^{\times}_2,y\mapsto \phi_2\phi_1(y).$$
Note that $\br_D(c)=\br_{D/Z}\br_Z(c)$ for $c\in A^D.$ 
  we can routinely compute that $\phi(\iota(c))=\iota'(\br_D(c))$  for all $c\in (A^D)^{\times}.$
\end{proof}

\subsection{Proof of theorem A}
Keep the notation of Theorem A.
Let $E=\Fp[\theta].$
At the beginning of this section, we have proved that $G=NH,C=N\cap H$ and $(\hH\times H)_{\theta}=(\hH\times H)_{\varphi}.$
Thus by Theorem \ref{thm:coho->H-tri}, we have $E=\Fp[\varphi].$
Let $X:N\ra \GL_m(E)=A^{\times}$ be a group representation affording $\theta$ and realised in $E,$ and let $A$ be the $G$-algebra associated with  $(G,N,\theta)_{\hH}$ and $X.$ 

{\it Step 1.} The restricted algebra $\Res^G_D(A)$ is a Dade $D$-algebra. Let $\br^A_D:A^D\ra B:=\bar{A}(D)$ be the Brauer homomorphism, where we identify $B$ with the matrix algebra $\mathrm{M}_l(E)$ for some positive integer $l$. 
Then the group representation $C\ra \GL_l(E)=B^{\times},c\mapsto \br^A_D(X(c))$ is absolutely irreducible and affords  $\varphi.$

Let $e_{\theta}\in FN$ and $e_{\varphi}\in FC$ be the block (with defect zero) idempotents corresponding to $\theta$ and $\varphi,$ respectively. 
Note that $e_{\theta}\in EN$ and $e_{\varphi}\in EC.$
The group algebra $EN$ becomes a $D$-algebra via the conjugation action of $D$ on $N$.
Consider the Brauer homomorphism
\[
\br_D^{EN} \colon (EN)^D \to EC
\]
which satisfies $\br_D^{EN}(c) = c$ for $c \in C,$ and $\br_D^{EN}(\tilde{n}) = 0$ for $n \in N \setminus C$, where $\tilde{n}$ denotes the sum of the $D$-conjugacy class of $n.$
The block $ENe_{\theta}$ is a direct summand of $EN$ and is isomorphic to a full matrix algebra over $E$. 
By the properties of the DGN correspondence, we have $\br^{EN}_D(e_{\theta}) = e_{\varphi}$.
This allows us to restrict the Brauer homomorphism to
\[
\br^{ENe_{\theta}}_D \colon (ENe_{\theta})^D \to ECe_{\varphi}.
\]
Since $ENe_{\theta}$ is a full matrix algebra with non-zero Brauer quotient $ECe_{\varphi}$, it follows that $ENe_{\theta}$ is a Dade $D$-algebra. 
Moreover, there exists a $D$-algebra isomorphism
\[
f \colon ENe_{\theta} \to \Res^G_D(A), \quad ne_{\theta} \mapsto X(n) \quad (n \in N).
\]
This induces an $E$-algebra isomorphism $h \colon ECe_{\varphi} \to B$ such that $\br^A_Df=h\br_D^{ENe_{\theta}}.$
The function $$X':C\ra\GL_l(E)=B^{\times},c\mapsto h(ce_{\varphi})$$ is a group representation affording $\varphi$ and realised in $E.$
For any $c\in C,$ we have $$X'(c)=h(ce_{\varphi})
=h\br_D^{ENe_{\theta}}(ce_{\theta})
=\br^A_Df(ce_{\theta})
=\br^A_D(X(c)).$$
This completes the proof of this step.

Let $\bar{G}=G/N=H/C.$ For any $g\in G,$ we denote its image in $\bar G$ by $\bar g\in\bar G.$

{\it Step 2.} 
Let $[\bar\alpha],[\bar\beta]\in\tH^2(\bar{G},E^{\times})$ be the cohomology class associated with $(G,N,\theta)_{\hH}$ and $(H,C,\varphi)_{\hH},$ respectively.
Then $[\bar\alpha]=[\bar\beta].$

Let $s = [E:\mathbb{F}_p]$ and fix a unitary embedding 
\(
\iota \colon A \hookrightarrow \mathfrak{E}_1 := \mathrm{M}_{ms}(\mathbb{F}_p).
\)
Let $\alpha:G\times G\ra E^{\times}$ be a factor set that is a lift of some $\bar\alpha\in[\bar\alpha],$ and let $Y:G\ra \fE_1^{\times}$ be a function satisfying the conditions of Theorem \ref{thm:coho} and such that $Y(g)Y(h)=Y(gh)\iota(\alpha(g,h))$ for all $g,h\in G.$
From {\it Step 1}, we have the representation 
\[
X' \colon C \to \mathrm{GL}_l(E) = B^\times, \quad c \mapsto \mathrm{br}^A_D(X(c)),
\]
which affords $\varphi$ and is realized in $E$. Let $B$ be the $H$-algebra associated with $(H,C,\varphi)_{\mathcal{H}}$ and $X'$. 
Fix a unitary embedding $\iota':B\ra\fE_2:=\M_{ls}(\Fp).$ 
By Theorem~\ref{thm:Dade}, there exists a group homomorphism
\[
\phi \colon \N_{\mathfrak{E}_1^\times}(\iota(D_1)) \to \mathfrak{E}_2^\times
\]
satisfying $\phi(\iota(y)) = \iota'(\mathrm{br}_D(y))$ for all $y \in (A^D)^\times$, where $D_1$ is the image of $D$ in $A^\times$.
For any $g \in H$, we have $Y(g) \in \N_{\mathfrak{E}_1^\times}(\iota(D_1))$, allowing us to define
\[
Y' \colon H \to \mathfrak{E}_2^\times, \quad g \mapsto \phi(Y(g)).
\]
For any $c\in C,$ we compute
\begin{align*}
  Y'(g)^{-1}\iota'(X'(c))Y'(g) &=\phi\big( Y(g)^{-1}\iota(X(c))Y(g)   \big) \\
   &= \phi\big( \iota(X(c^g)) \big)
   =\iota'\big( \br^A_D(X(c^g)) \big)=\iota'(X'(c^g)).
\end{align*}
Similarly, for $z \in E$ with $z^g = z^{\sigma_g}$ (where $\sigma_g \in \mathcal{H}$ satisfies $\theta^{\sigma_g} = \theta^{g^{-1}}$), we have
\[
Y'(g)^{-1}\iota'(z \cdot 1_B)Y'(g) = \iota'(z^g \cdot 1_B).
\]
Since $B$ is $E$-linearly spanned by $\{X'(c) \mid c \in C\}$, we have
\[
Y'(g)^{-1}\iota'(x)Y'(g) = \iota'(x^g) \quad \text{for all } x \in B.
\]
It is routine to check that $Y'(c)=\iota'(X'(c)),Y'(gc)=Y'(g)Y'(c),Y'(cg)=Y'(c)Y'(g)$ for all $g\in H$ and $c\in C.$
Let $\beta \colon H \times H \to E^\times$ be the factor set defined by 
\begin{equation}\label{eq:factor-relation}
Y'(g)Y'(h) = Y'(gh)\iota'(\beta(g,h)).
\end{equation}
Direct computation shows
\[
Y'(g)Y'(h) = \phi(Y(g)Y(h))
= \phi(Y(gh)\iota(\alpha(g,h)))
= Y'(gh)\iota'(\alpha(g,h)).
\]
Comparing this with \eqref{eq:factor-relation}, we conclude that 
\[
\alpha(g,h) = \beta(g,h) \quad \text{for all } g,h \in H,
\]
which completes the proof of this step.

Since $[\bar{\alpha}] = [\bar{\beta}]$, an application of Theorem~\ref{thm:coho->H-tri} yields 
\[
(G,N,\theta)_{\mathcal{H}} \geqslant (H,C,\varphi)_{\mathcal{H}}.
\]
Theorem~A now follows immediately from Theorem~\ref{thm:H-tri}.

\subsection{Proof of Corollary B}
Keep the notation of Corollary B.
Let $T=\Aut(G)_{\theta^{\hH},D},$ and let $\widetilde{G}= G\semi T$ be the semidirect product of $G$ and $T.$ 
Observe that $M$ is normal in $\widetilde{G}$ and $\N_{\widetilde{G}}(D)=\N_G(D)\semi T.$ 
By Theorem A, there exists an $\hH\times \N_{\widetilde{G}}(D)$-equivariant bijection $$\varDelta_G:\IBr(G\which \theta^{\hH})\ra\IBr(\N_G(D)\which \varphi^{\hH}).$$ 
Since $\big(\mathcal{H} \times \mathrm{Aut}(G)\big)_{\theta,D}$ embeds naturally as a subgroup of $\mathcal{H} \times \mathrm{N}_{\widetilde{G}}(D)$, the map $\Delta_G$ is automatically $\big(\mathcal{H} \times \mathrm{Aut}(G)\big)_{\theta,D}$-equivariant.
Furthermore, conditions (2) and (3) of Theorem~A imply that $\Delta_G$ restricts to an $\big(\mathcal{H} \times \mathrm{Aut}(G)\big)_{\theta,D}$-equivariant bijection
\[
f \colon \mathrm{IBr}(G \mid \theta) \to \mathrm{IBr}(\mathrm{N}_G(D) \mid \varphi).
\]
The statement about vertices follows immediately from condition (4) of Theorem~A.

\noindent\textbf{Acknowledgement.} The author is deeply grateful to Yuanyang Zhou for his invaluable suggestions and to Xueqin Hu for kindly introducing Dade’s theorem to me.
 The author also wishes to thank the anonymous referee for their meticulous reading of the manuscript and their insightful comments, which have greatly improved the quality of this paper.

\end{document}